\documentclass[12pt, reqno]{amsart}
\usepackage{amsmath, amstext, amsbsy, amssymb}

\makeatletter
\@namedef{subjclassname@2020}{%
  \textup{2020} Mathematics Subject Classification}
\makeatother

\newtheorem{theorem}{Theorem}[section]
\newtheorem{lemma}[theorem]{Lemma}
\newtheorem{proposition}[theorem]{Proposition}
\newtheorem{corollary}[theorem]{Corollary}
\theoremstyle{definition}
\newtheorem{definition}[theorem]{Definition}

\theoremstyle{remark}
\newtheorem{remark}[theorem]{Remark}
\newtheorem{conjecture}[theorem]{Conjecture}

\numberwithin{equation}{section}

\newcommand{\ch}{{\rm ch} }
\newcommand{\C}{ \mathbb C }
\newcommand{\Coe}{ {\rm Coeff} }

\newcommand{\End}{{\rm End}}
\newcommand{\fa}{ \mathfrak a }

\newcommand{\fG}{ \mathfrak G }
\newcommand{\fL}{ \mathfrak L }
\newcommand{\fn}{ \mathfrak n }
\newcommand{\fock}{{\mathbb H}_X}
\newcommand{\fockprime}{{\mathbb H}_X'}

\newcommand{\Hn}{H^*(\Xn)}

\newcommand{\la}{{\lambda}}
\newcommand{\lambsq}{s(\lambda)}
\newcommand{\Ln}{L^{[n]}}
\newcommand{\MD}{\mathcal {MD}}

\newcommand{\N}{\mathbb N}
\newcommand{\On}{{\mathcal O}^{[n]}}

\newcommand{\Q}{ \mathbb Q }

\newcommand{\qMD}{{\rm q}{\mathcal {MD}}}
\newcommand{\qMZV}{ {\bf qMZV} }

\newcommand{\T}{ \mathbb T }
\newcommand{\Tr}{ {\rm Tr} }

\newcommand{\vac}{|0\rangle}

\newcommand{\w}{\tilde}
\newcommand{\W}{\widetilde}
\newcommand{\Wb}{ {\bf W} }

\newcommand{\Xn}{ {X^{[n]}}}
\newcommand{\Z}{ \mathbb Z }

\def\beq{\begin{equation}}
\def\eeq{\end{equation}}

\numberwithin{equation}{section}
\begin{document}

\title[Toward Qin's Conjecture]
      {Toward Qin's Conjecture on Hilbert schemes of points 
      and quasi-modular forms}
      
%\begin{aug}
%   \author{\fnms{Mazen M.} \snm{Alhwaimel}\ead[label=e1]{4200@qu.edu.sa}}
%   \address{Department of Mathematics\\
 %           College of Science\\
 %           Qassim University\\
 %          P. O. Box 6644\\
  %         Buraydah 51452\\
  %         Saudi Arabia\\\printead{e1}}
%\end{aug}

%\author[]{Mazen M. Alhwaimel}

\author[Mazen M. Alhwaimel]{Mazen M. Alhwaimel}
\address{Department of Mathematics, College of Science, Qassim University, P. O. Box 6644, Buraydah 51452, Saudi Arabia}
\email{4200@qu.edu.sa} 

\date{\today}

\begin{abstract}
For a line bundle $L$ on a smooth projective surface $X$
and nonnegative integers $k_1, \ldots, k_N$, 
Okounkov \cite{Oko} introduced the reduced generating series 
$\big \langle \ch_{k_1}^{L} \cdots \ch_{k_N}^{L} \big \rangle'$ 
for the intersection numbers among the Chern characters of 
the tautological bundles over the Hilbert schemes of points on $X$ 
and the total Chern classes of the tangent bundles of 
these Hilbert schemes. 
In \cite{Qin2}, Qin conjectured  that 
these reduced generating series are quasi-modular forms if 
the canonical divisor of $X$ is numerically trivial. 
In this paper, we verify that Qin's conjecture holds  for 
$\langle \ch_1^{L_1}\ch_1^{L_2} \rangle'$.
The main approaches are to use the methods laid out in \cite{QY} and construct various relations regarding
 multiple $q$-zeta values and quasi-modular forms.

\end{abstract}

\maketitle

%\tableofcontents
%%
%%
%%
%%
%%
%%
%%
%%
%%
%%
%%
%%
%%
\section{\bf Introduction} 
\label{sect_intr}

Hilbert schemes of points have been of great interests in algebraic geometry since the pioneering work of Grothendieck.
Hilbert schemes of points, parametrizing $0$-dimensional closed subschemes,
on a smooth projective surface are known to be smooth and irreducible. 

Let $X$ be a smooth projective complex surface, 
and let $\Xn$ be the Hilbert scheme of $n$ points in $X$. A line bundle 
$L$ on $X$ induces a tautological rank-$n$ bundle $\Ln$ on $\Xn$.
Let $\ch_k(\Ln)$ be the $k$-th Chern character of $\Ln$.
Following Okounkov \cite{Oko}, we introduce the two generating series:
\begin{eqnarray}     
\big \langle \ch_{k_1}^{L_1} \cdots \ch_{k_N}^{L_N} \big \rangle
  &=&\sum_{n \ge 0} q^n \, \int_\Xn \ch_{k_1}(L_1^{[n]}) \cdots \ch_{k_N}(L_N^{[n]}) 
        \cdot c(T_\Xn)   \label{OkoChkN.1}   \\
\big \langle \ch_{k_1}^{L_1} \cdots \ch_{k_N}^{L_N} \big \rangle'
  &=&\frac{\big \langle \ch_{k_1}^{L_1} \cdots \ch_{k_N}^{L_N} 
     \big \rangle}{\langle \rangle}
      = (q; q)_\infty^{\chi(X)} \cdot \big \langle 
      \ch_{k_1}^{L_1} \cdots \ch_{k_N}^{L_N} \big \rangle  \label{OkoChkN.2} 
\end{eqnarray}
where $0 < |q| < 1$, $c\big (T_\Xn \big )$ is the total Chern class of 
the tangent bundle $T_\Xn$, $\chi(X)$ is the Euler characteristics of $X$, 
$(a; q)_n = \prod_{i=0}^n (1-aq^i)$, and  
$$
  \langle \rangle 
= \sum_{n \ge 0} q^n \, \int_\Xn c(T_\Xn)
= \frac1{(q; q)_\infty^{\chi(X)}}
$$
which is a formula due to G\"ottsche \cite{Got}.

\begin{conjecture}   \label{OkoConj}
(\cite[Conjecture~2]{Oko}) Let $L$ be a line bundle on a smooth 
projective surface $X$. Then 
$\big \langle \ch_{k_1}^L \cdots \ch_{k_N}^L \big \rangle'$ is 
a multiple $q$-zeta value of weight at most 
$
\sum_{i=1}^N (k_i + 2).
$
\end{conjecture}

\begin{conjecture}   \label{QinConj}
(\cite{Qin2}) Let $L_1, \ldots, L_N$ be line bundles on a smooth 
projective surface $X$. If the canonical divisor of $X$ 
is numerically trivial, then 
$\big \langle \ch_{k_1}^{L_1} \cdots \ch_{k_N}^{L_N} \big \rangle'$ is  
a quasi-modular form of weight at most $\sum_{i=1}^N (k_i + 2)$.
\end{conjecture}

In the region ${\rm Re} \, s > 1$, the Riemann zeta function is defined by 
$$
\zeta(s) = \sum_{n =1}^\infty \frac{1}{n^{s}}.
$$
The integers $s > 1$ give rise to a sequence of special values of the Riemann zeta function.
Multiple zeta values are series of the form
$$
\zeta(s_1, \ldots, s_k) = \sum_{n_1 > \cdots > n_k > 0} 
\frac{1}{n_1^{s_1} \cdots n_k^{s_k}}
$$
where $s_1, \ldots, s_k$ are positive integers with $s_1 > 1$, 
and $n_1, \ldots, n_k$ denote positive integers.
%and $\{s_, \ldots, s_k \}$ is called admissible. 
Multiple $q$-zeta values are $q$-deformations of 
$\zeta(s_1, \ldots, s_k)$, which may take different forms 
\cite{Bac, BK2, Bra1, Bra2, OT, Zhao, Zud}. 
The multiple $q$-zeta values defined by Okounkov \cite{Oko} 
are denoted by $Z(s_1, \ldots, s_k)$ where $s_1, \ldots, s_k > 1$ 
are integers (see Definition~\ref{def_qMZV}~(iv)). 
For instance, we have

$$
Z(2) = \sum_{n>0} \frac{q^n}{(1-q^n)^2}, \quad
Z(3) = \sum_{n>0} \frac{q^n(q^n+1)}{(1-q^n)^3}, 
$$
\begin{eqnarray}   \label{exampleZs}
Z(4) = \sum_{n>0} \frac{q^{2n}}{(1-q^n)^4}, \quad
Z(6) = \sum_{n>0} \frac{q^{3n}}{(1-q^n)^6}.
\end{eqnarray}

The weight of $Z(s_1, \ldots, s_k)$ is defined to be 
$s_1 + \ldots + s_k$. 
By \cite[Theorem~2.4]{BK3}, the $\Q$-linear span $\qMZV$ of 
all the multiple $q$-zeta values $Z(s_1, \ldots, s_k)$ with 
$s_1, \ldots, s_k > 1$ is an algebra over $\Q$. 
By ~(\ref{Z246}), the set ${\bf QM}$ of all quasi-modular forms 
(of level $1$ on the full modular group ${\rm PSL}(2; \Z)$) over $\Q$ 
is a subalgebra of $\qMZV$:
$$
{\bf QM} = \Q\big [Z(2), Z(4), Z(6) \big ] \subset \qMZV.
$$

The main result of our paper is the following theorem about $\langle \ch_1^{L_1}\ch_1^{L_2} \rangle'$.

\begin{theorem} \label{intro_corollary_ch1Lch1L}
Let $L_1$ and $L_2$ be line bundles over a smooth projective surface $X$.
If the canonical class of $X$ is numerically trivial, 
then $\langle \ch_1^{L_1}\ch_1^{L_2} \rangle'$ is equal to
$$
\left (\frac72 Z(4) - \frac12 Z(2)^2 + Z(2) \right ) \cdot 
\langle L_1, L_2 \rangle
$$
$$
+ \left (\frac{5}{4} Z(2)^2 + \frac{5}{4} Z(4) - \frac{10}{3} Z(2)^3 
+ 5 Z(2)Z(4) + \frac{35}{6} Z(6) \right ) \cdot \chi(X).
$$
In particular, Qin's Conjecture~\ref{QinConj} holds for 
$\langle \ch_1^{L_1}\ch_1^{L_2} \rangle'$.
\end{theorem}

To prove Theorem~\ref{intro_corollary_ch1Lch1L}, 
following ~\cite{Qin1}, we first study the reduced generating series 
$$
  F^{\alpha_1, \ldots, \alpha_N}_{k_1, \ldots, k_N}(q) 
= \sum_{n \ge 0} q^n \int_\Xn \left ( \prod_{i=1}^N 
  G_{k_i}(\alpha_i, n) \right ) c\big (T_\Xn \big )
$$
for cohomology classes $\alpha_1, \ldots, \alpha_N \in H^*(X)$ 
and integers $k_1, \ldots, k_N \ge 0$, 
where $G_{k}(\alpha, n)$ is the homogeneous component in 
$H^{|\alpha|+2k}(\Xn)$ of \eqref{DefOfGGammaN} 
for a homogeneous class $\alpha \in H^*(X)$. 
By Lemma~\ref{FtoW},  
$F^{\alpha_1, \ldots, \alpha_N}_{k_1, \ldots, k_N}(q)$ can be 
expressed in terms of the Ext vertex operators of 
Carlsson and Okounkov \cite{Car1, Car2, CO} and 
the Chern character operators $\mathfrak G_{k_i}(\alpha_i)$ 
which act on $H^*(\Xn)$ by the cup product with the classes
$G_{k_i}(\alpha_i, n)$.  We have

 \begin{eqnarray}   \label{ch1Lch1L.0}
   \langle \ch_1^{L_1} \ch_1^{L_2} \rangle'
= (q;q)_\infty^{\chi(X)} \cdot \big (F_{1,1}^{1_X, 1_X}(q)
  + F_{1,0}^{1_X, L_1}(q) + F_{1,0}^{1_X, L_2}(q)
  + F_{0,0}^{L_1, L_2}(q) \big )
\end{eqnarray}
which is the central part of the final section. We devote the rest of this paper to
explicitly compute and establish Qin's Conjecture~\ref{QinConj}.

\

The paper is organized as follows. Section~\ref{sect_qMZV} is where we give some results about multiple 
$q$-zeta values and quasi-modular forms. Section~\ref{sect_Hilbert} summarizes some basic
 results regarding Hilbert schemes of points on smooth projective surfaces,
 the Heisenberg operators of Grojnowski and Nakajima, and the Chern character operators.
 In Section~\ref{sect_CO}, 
we introduce the generating series 
$F^{\alpha_1, \ldots, \alpha_N}_{k_1, \ldots, k_N}(q)$ 
in terms of the Chern character operators.
Section~\ref{sect_CompEquiv} calculates the reduced series 
 $\langle \ch_1 \ch_1 \rangle'$ in the equivariant setting.
 Theorem~\ref{intro_corollary_ch1Lch1L} (= Theorem~\ref{corollary_ch1Lch1L}) 
is verified in Section~\ref{sect_ch1L1ch1L2}.

\bigskip\noindent
%{\bf Acknowledgment.}
%%
%%
%%
%%
%%
%%
%%
%%
%%
%%
%%
%%
\section{\bf Multiple $q$-zeta values} 
\label{sect_qMZV}

In this section, we will recall basic definitions and facts 
regarding multiple $q$-zeta values and quasi-modular forms.

The following definitions and notations are from \cite{BK1, BK3, Oko}.

\begin{definition}  \label{def_qMZV}
Let $\N = \{1,2,3, \ldots\}$, and fix a subset $S \subset \N$. 
\begin{enumerate}
\item[{\rm (i)}]
Let $Q = \{Q_s(t)\}_{s \in S}$ where each $Q_s(t) \in \Q[t]$ is a polynomial 
with $Q_s(0) = 0$ and $Q_s(1) \ne 0$. For $s_1, \ldots, s_\ell \in S$ 
with $\ell \ge 1$, define
$$
Z_Q(s_1, \ldots, s_\ell)
= \sum_{n_1 > \cdots > n_\ell \ge 1} \prod_{i=1}^\ell 
  \frac{Q_{s_i}(q^{n_i})}{(1-q^{n_i})^{s_i}}
\in \Q[[q]].
$$
Put $Z_Q(\emptyset) = 1$, and define $Z(Q, S)$ to be the $\Q$-linear span 
of the set
$$
\{Z_Q(s_1, \ldots, s_\ell)| \, \ell \ge 0 \text{ and }
s_1, \ldots, s_\ell \in S \}.
$$

\item[{\rm (ii)}]
Define $\mathcal {MD} = Z(Q^E, \N)$ where $Q^E = \{Q_s^E(t)\}_{s \in \N}$
and 
$$
Q_s^E(t) = \frac{tP_{s-1}(t)}{(s-1)!}
$$ 
with $P_s(t)$ being the Eulerian polynomial defined by 
\begin{eqnarray}   \label{def_qMZV.01}
\frac{tP_{s-1}(t)}{(1-t)^s} = \sum_{d=1}^\infty d^{s-1}t^d.
\end{eqnarray}  
We have $t P_{0}(t) = t$.
For $s > 1$, the polynomial $t P_{s-1}(t)$ has degree $s - 1$. 
For $s_1, \ldots, s_\ell \in \N$ with $\ell \ge 1$, 
define the Bachmann-K\" uhn series
\begin{eqnarray}   \label{def_qMZV.02}
[s_1, \ldots, s_\ell] 
= Z_{Q^E}(s_1, \ldots, s_\ell)
= \sum_{n_1 > \cdots > n_\ell \ge 1} \prod_{i=1}^\ell 
  \frac{Q_{s_i}^E(q^{n_i})}{(1-q^{n_i})^{s_i}}.
\end{eqnarray}  

\item[{\rm (iii)}]
Define $\text{q}\mathcal {MD}$ be the subspace of $\mathcal {MD}$ linearly 
spanned by $1$ and all the brackets $[s_1, \ldots, s_\ell]$ with $s_1 > 1$. 

\item[{\rm (iv)}]
Define $\qMZV = Z(Q^O, \N_{>1})$ where 
$Q^O = \{Q_s^O(t)\}_{s \in \N_{>1}}$ and 
$$
  Q_s^O(t) 
= \begin{cases}
  t^{s/2}            &\text{if $s \ge 2$ is even;} \\
  t^{(s-1)/2} (t+1)  &\text{if $s \ge 3$ is odd.}
  \end{cases}
$$ 
For $s_1, \ldots, s_\ell \in \N$ with $\ell \ge 1$, 
define the Okounkov series
$$
Z(s_1, \ldots, s_\ell) = Z_{Q^O}(s_1, \ldots, s_\ell).
$$
\end{enumerate}
\end{definition}

We see from Definition~\ref{def_qMZV}~(ii) that for $s \ge 1$, 
\begin{eqnarray}   \label{def_qMZV.03}
(s-1)! \cdot \frac{Q_{s}^E(t)}{(1-t)^s} = \sum_{d=1}^\infty d^{s-1}t^d.
\end{eqnarray} 
It follows that 
\begin{eqnarray*}   
  (s-1)! \cdot [s] 
= (s-1)! \cdot \sum_{n \ge 1} \frac{Q_{s}^E(q^n)}{(1-q^n)^s} 
= \sum_{n \ge 1} \sum_{d=1}^\infty d^{s-1} q^{nd}
\end{eqnarray*} 
for $s \ge 1$. Therefore, we obtain
\begin{eqnarray}   \label{def_qMZV.04}
  [s] 
= \frac{1}{(s-1)!} \sum_{n, d \ge 1} d^{s-1} q^{nd}
= \frac{1}{(s-1)!} \sum_{d \ge 1} d^{s-1} \frac{q^{d}}{1-q^d}.
\end{eqnarray}  
More generally, for $s_1, \ldots, s_\ell \ge 1$, we have
\begin{eqnarray}   \label{def_qMZV.05}
  [s_1, \ldots, s_\ell] 
= \frac{1}{(s_1-1)! \cdots (s_\ell-1)!} 
  \sum_{\substack{n_1 > \cdots > n_\ell \ge 1\\d_1, \ldots, d_\ell \ge 1}} 
  d_1^{s_1-1} \cdots d_\ell^{s_\ell-1} q^{n_1d_1+\ldots+n_\ell d_\ell}.
\end{eqnarray} 

Some examples of the Bachmann-K\" uhn series $[s], s \ge 1$ are
\begin{eqnarray}   \label{example[s]}
[1] = \sum_{n>0} \frac{q^n}{1-q^n}, \quad
[2] = \sum_{n>0} \frac{nq^n}{1-q^n}, \quad
[3] = \frac12 \sum_{n>0} \frac{n^2 q^n}{1-q^n}.
\end{eqnarray}
Some examples of the Okounkov series $Z(s), s \ge 1$ are given by 
\eqref{exampleZs}.

By the Theorem~2.13 and Theorem~2.14 in \cite{BK1},
$\text{q}\mathcal {MD}$ is a subalgebra of $\mathcal {MD}$, 
and $\mathcal {MD}$ is a polynomial ring over $\text{q}\mathcal {MD}$ 
with indeterminate $[1]$:
\begin{eqnarray}   \label{BK1Thm214}
\mathcal {MD} = \text{q}\mathcal {MD}[\,[1]\,].
\end{eqnarray}
By the Proposition~2.2 and Theorem~2.4 in \cite{BK3}, 
$Z(\{Q_s^E(t)\}_{s \in \N_{>1}}, \N_{>1})$ is a subalgebra of 
$\mathcal {MD}$ as well and 
$\qMZV = Z(\{Q_s^E(t)\}_{s \in \N_{>1}}, \N_{>1})$. Therefore,   
\begin{eqnarray}   \label{BK3-2.4}
\qMZV = Z(\{Q_s^E(t)\}_{s \in \N_{>1}}, \N_{>1}) \subset \qMD \subset \MD
\end{eqnarray} 
are inclusions of $\Q$-algebras.  
For instance, by \cite[Example~2.6]{BK3}, 
\begin{eqnarray}   \label{BK3-2.6}
Z(2) = [2], \quad Z(3) = 2[3], \quad Z(4) = [4] - \frac16 [2].
\end{eqnarray}

The graded ring ${\bf QM}$ of quasi-modular forms (of level $1$  
on the full modular group ${\rm PSL}(2; \Z)$) over $\Q$ is 
the polynomial ring over $\Q$ generated by the Eisenstein series 
$G_2(q), G_4(q)$ and $G_6(q)$:
$$
{\bf QM} = \Q[G_2, G_4, G_6] = {\bf M}[G_2]
$$
where ${\bf M} = \Q[G_4, G_6]$ is the graded ring of modular forms 
(of level $1$) over $\Q$, and 
\begin{eqnarray*}    
G_{2k} = G_{2k}(q) =\frac{1}{(2k-1)!} \cdot 
\left (-\frac{B_{2k}}{4k} + \sum_{n \ge 1} 
\Big ( \sum_{d|n} d^{2k-1} \Big )q^n \right )
\end{eqnarray*}
where $B_i \in \Q, i \ge 2$ are the Bernoulli numbers defined by 
$$
\frac{t}{e^t - 1} = 1 - \frac{t}{2} + \sum_{i =2}^{+\infty} B_i \cdot \frac{t^i}{i!}.
$$
%where $B_i$ denotes the Bernoulli numbers with $B_2=1/6, B_4 = -1/30, B_6 = 1/42$, etc. 
The grading is to assign $G_2, G_4, G_6$ weights $2, 4, 6$ respectively. 
%For a nonnegative integer $k$, let $d_{2k}$ be the dimension of 
%the vector space consisting of weight-$2k$ quasi-modular forms. 
%By \cite[Proposition~2.7]{BK3},
%\begin{eqnarray}  \label{BK3Prop2.7} 
%  \sum_{k \ge 0} d_{2k} x^{2k} 
%= \left (1 + \frac{x^4}{1-x^2} + \frac{x^{12}}{(1-x^4)(1-x^6)} \right )
%  \frac1{1-x^2}.
%\end{eqnarray}
By \cite[p.6]{BK3}, 
\begin{eqnarray*}
G_2 &=& -\frac{1}{24} + Z(2),      \\
G_4 &=& \frac{1}{1440} + Z(2) + \frac{1}{6}Z(4),     \\
G_6 &=& -\frac{1}{60480} + \frac{1}{120} Z(2) + \frac{1}{4} Z(4) + Z(6).
\end{eqnarray*}
It follows that
\begin{eqnarray}  \label{Z246} 
{\bf QM} = \Q[Z(2), Z(4), Z(6)].
\end{eqnarray}
\section{\bf Hilbert schemes of points on surfaces} 
\label{sect_Hilbert}

Let $X$ be a smooth projective complex surface,
and $\Xn$ be the Hilbert scheme of $n$ points in $X$. 
An element in $\Xn$ is represented by a
length-$n$ $0$-dimensional closed subscheme $\xi$ of $X$. For $\xi
\in \Xn$, let $I_{\xi}$ be the corresponding sheaf of ideals. It
is well known that $\Xn$ is smooth. 
%Sending an element in $\Xn$ to its support in the symmetric product ${\rm Sym}^n(X)$, 
%we obtain the Hilbert-Chow morphism $\pi_n: \Xn \rightarrow {\rm Sym}^n(X)$,
%which is a resolution of singularities. 
Define the universal codimension-$2$ subscheme:
\begin{eqnarray*}
{ \mathcal Z}_n=\{(\xi, x) \subset \Xn\times X \, | \, x\in
 {\rm Supp}{(\xi)}\}\subset \Xn\times X.
\end{eqnarray*}
Denote by $p_1$ and $p_2$ the projections of $\Xn \times X$ to
$\Xn$ and $X$ respectively. A line bundle $L$ on $X$ induces 
the tautological rank-$n$ bundle $\Ln$ over $\Xn$:
\begin{eqnarray}    \label{L[n]}
\Ln = (p_{1}|_{\mathcal Z_n})_*\big (p_2^*L|_{\mathcal Z_n} \big ).
\end{eqnarray}
Let
\begin{eqnarray*}
\fock = \bigoplus_{n=0}^\infty \Hn
\end{eqnarray*}
be the direct sum of total cohomology groups of the Hilbert schemes $\Xn$.
For $m \ge 0$ and $n > 0$, let $Q^{[m,m]} = \emptyset$ and define
$Q^{[m+n,m]}$ to be the closed subset:
%For $n \ge 0$, define $Q^{[n,n]} = \emptyset$. For $n \ge 0$ and
%$\ell > 0$, define $Q^{[n+\ell,n]} \subset X^{[n+\ell]} \times X
%\times \Xn$ to be the following closed subset:
$$\{ (\xi, x, \eta) \in X^{[m+n]} \times X \times X^{[m]} \, | \,
\xi \supset \eta \text{ and } \mbox{Supp}(I_\eta/I_\xi) = \{ x \}
\}.$$

We recall Nakajima's definition of the Heisenberg operators \cite{Nak}. 
Let $\alpha \in H^*(X)$. Set $\mathfrak a_0(\alpha) =0$.
For $n > 0$, the operator $\mathfrak
a_{-n}(\alpha) \in \End(\fock)$ is
defined by
$$
\mathfrak a_{-n}(\alpha)(a) = \w{p}_{1*}([Q^{[m+n,m]}] \cdot
\w{\rho}^*\alpha \cdot \w{p}_2^*a)
$$
for $a \in H^*(X^{[m]})$, where $\w{p}_1, \w{\rho},
\w{p}_2$ are the projections of $X^{[m+n]} \times X \times
X^{[m]}$ to $X^{[m+n]}, X, X^{[m]}$ respectively. Define
$\mathfrak a_{n}(\alpha) \in \End(\fock)$ to be $(-1)^n$ times the
operator obtained from the definition of $\mathfrak
a_{-n}(\alpha)$ by switching the roles of $\w{p}_1$ and $\w{p}_2$. 
We often refer to $\mathfrak a_{-n}(\alpha)$ (resp. $\mathfrak a_n(\alpha)$) 
as the {\em creation} (resp. {\em annihilation})~operator. 
The following is from \cite{Nak, Gro}. Our convention of the sign follows \cite{LQW2}.

\begin{theorem} \label{commutator}
The operators $\mathfrak a_n(\alpha)$ satisfy
the commutation relation:
\begin{eqnarray*}
\displaystyle{[\mathfrak a_m(\alpha), \mathfrak a_n(\beta)] = -m
\; \delta_{m,-n} \cdot \langle \alpha, \beta \rangle \cdot {\rm Id}_{\fock}}.
\end{eqnarray*}
The space $\fock$ is an irreducible module over the Heisenberg
algebra generated by the operators $\mathfrak a_n(\alpha)$ with a
highest~weight~vector $\vac=1 \in H^0(X^{[0]}) \cong \C$.
\end{theorem}

The Lie bracket in the above theorem is understood in the super
sense according to the parity of the degrees of the
cohomology classes involved. It follows from
Theorem~\ref{commutator} that the space $\fock$ is linearly
spanned by all the Heisenberg monomials $\mathfrak
a_{n_1}(\alpha_1) \cdots \mathfrak a_{n_k}(\alpha_k) \vac$
where $k \ge 0$ and $n_1, \ldots, n_k < 0$.

\begin{definition} \label{partition}
\begin{enumerate}
\item[{\rm (i)}]
A {\it generalized partition} \index{partition, generalized} of an integer $n$ is of the form
$$
\lambda = (\cdots  (-2)^{m_{-2}}(-1)^{m_{-1}} 1^{m_1}2^{m_2} \cdots)
$$ 
such that part $i\in \Z$ has multiplicity $m_i$, $m_i \ne 0$ for 
only finitely many $i$'s, and $n = \sum_i i m_i$. Define 
$\ell(\la) = \sum_i m_i$, $|\la| = \sum_i im_i$, 
$|\la|_+ = \sum_{i > 0} im_i$, $s(\la) = \sum_i i^2m_i$,
and $\lambda^! = \prod_i m_i!$.

\item[{\rm (ii)}] 
Let $\alpha \in H^*(X)$ and $k \ge 1$. Define $\tau_{k*}: H^*(X) \to H^*(X^k)$ 
to be the linear map induced by the diagonal embedding $\tau_k: X \to X^k$, and
$$
(\mathfrak a_{m_1} \cdots \mathfrak a_{m_k})(\alpha)
= \mathfrak a_{m_1} \cdots \mathfrak a_{m_k}(\tau_{*k}\alpha)
= \sum_j \mathfrak a_{m_1}(\alpha_{j,1}) \cdots \mathfrak a_{m_k}(\alpha_{j,k})
$$ 
when $\tau_{k*}\alpha = \sum_j \alpha_{j,1} \otimes \cdots 
\otimes \alpha_{j, k}$ via the K\"unneth decomposition of $H^*(X^k)$.

\item[{\rm (iii)}]
For a generalized partition $\lambda = 
\big (\cdots (-2)^{m_{-2}}(-1)^{m_{-1}} 1^{m_1}2^{m_2} \cdots \big )$, define
\begin{eqnarray*}
\mathfrak a_{\lambda}(\alpha) = \left ( \prod_i (\mathfrak
a_i)^{m_i} \right ) (\alpha)
\end{eqnarray*}
where the product $\prod_i (\mathfrak
a_i)^{m_i} $ is understood to be
$\cdots \mathfrak a_{-2}^{m_{-2}} \mathfrak a_{-1}^{m_{-1}}
 \mathfrak a_{1}^{m_{1}} \mathfrak a_{2}^{m_{2}}\cdots$.
\end{enumerate}
\end{definition}

By \cite[Lemma~3.12~(i)]{Qin1}, the commutator 
$[\mathfrak a_{n_1} \cdots \mathfrak a_{n_{k}} (\alpha),
\mathfrak a_{m_1} \cdots \mathfrak a_{m_{s}}(\beta)]$ is equal to
\begin{eqnarray}    \label{QinLemma3.12(i)}
-\sum_{t=1}^k \sum_{j=1}^s n_t \delta_{n_t, -m_j} \cdot
\left ( \prod_{\ell=1}^{j-1} \mathfrak a_{m_\ell} \cdot
\prod_{1 \le u \le k, u \ne t} \mathfrak a_{n_u} \cdot
\prod_{\ell=j+1}^{s} \mathfrak a_{m_\ell} \right )(\alpha\beta).
\end{eqnarray}

For $n > 0$ and a homogeneous class $\alpha \in H^*(X)$, let
$|\alpha| = s$ if $\alpha \in H^s(X)$, and let $G_k(\alpha, n)$ be
the homogeneous component in $H^{|\alpha|+2k}(\Xn)$ of
\begin{eqnarray}    \label{DefOfGGammaN}
 G(\alpha, n) = p_{1*}(\ch({\mathcal O}_{{\mathcal Z}_n}) \cdot p_2^*\alpha
\cdot p_2^*{\rm td}(X) ) \in \Hn
\end{eqnarray}
where $\ch({\mathcal O}_{{\mathcal Z}_n})$ denotes the Chern
character of ${\mathcal O}_{{\mathcal Z}_n}$
and ${\rm td}(X)$ denotes the Todd class. We extend the notion $G_k(\alpha,
n)$ linearly to an arbitrary class $\alpha \in H^*(X)$, 
and set $G(\alpha, 0) =0$. 
It was proved in \cite{LQW1} that the cohomology ring of $\Xn$ is
generated by the classes $G_{k}(\alpha, n)$ where $0 \le k < n$
and $\alpha$ runs over a linear basis of $H^*(X)$. 
The {\it Chern character operator} ${\mathfrak G}_k(\alpha) \in
\End({\fock})$ is the operator acting on $H^*(\Xn)$ by the cup product with $G_k(\alpha, n)$. 

The following result is from \cite{LQW2} (see also \cite{Qin1}).

\begin{theorem} \label{char_th}
Let $k \ge 0$ and $\alpha, \beta \in H^*(X)$. 
Then, $\mathfrak G_k(\alpha)$ is equal to
\begin{eqnarray*}
& &- \sum_{\ell(\lambda) = k+2, |\lambda|=0}
   \frac{1}{\lambda^!} \mathfrak a_{\lambda}(\alpha)
   + \sum_{\ell(\lambda) = k, |\lambda|=0}
   \frac{\lambsq - 2}{24\lambda^!}
   \mathfrak a_{\lambda}(e_X \alpha)  \\
&+&\sum_{\ell(\lambda) = k+1, |\lambda|=0} \frac{g_{1, \lambda}}{\la^!}
   \mathfrak a_{\lambda}(K_X \alpha)
   + \sum_{\ell(\lambda) = k, |\lambda|=0} \frac{g_{2, \lambda}}{\la^!}
   \mathfrak a_{\lambda}(K_X^2 \alpha)
\end{eqnarray*}
where all the numbers $g_{1, \lambda}$ and $g_{2, \lambda}$ are 
independent of $X$ and $\alpha$. 
\end{theorem}

It follows that 
\begin{eqnarray}  \label{fG0L}
\fG_0(L) = -\sum_{m > 0} \fa_{-m}\fa_m(L). 
\end{eqnarray}
Moreover, by \cite[Proposition~4.8]{Qin1}, we have
\begin{eqnarray}  \label{char_thK=1.0}
 \mathfrak G_1(\alpha) 
= - \sum_{\ell(\lambda) = 3, |\lambda|=0}
   \frac{1}{\lambda^!} \mathfrak a_{\lambda}(\alpha)
  + \sum_{n>0} \frac{1-n}2 \mathfrak a_{-n}\mathfrak a_n (K_X \alpha).
\end{eqnarray}

\section{\bf The generating series 
$F^{\alpha_1, \ldots, \alpha_N}_{k_1, \ldots, k_N}(q)$} 
\label{sect_CO}

In this section, we will define the generating series 
$F^{\alpha_1, \ldots, \alpha_N}_{k_1, \ldots, k_N}(q)$ in terms of 
the Chern character operators,
and rewrite it by using the Ext vertex operators constructed 
by Carlsson and Okounkov \cite{CO, Car1, Car2}.

%Equip the smooth projective surface $X$ with a trivial $\C^*$ action.
Let $L$ be a line bundle over the smooth projective surface $X$. 
Let $\mathbb E_L$ be the virtual vector bundle on $X^{[k]} \times X^{[\ell]}$
whose fiber at $(I, J) \in X^{[k]} \times X^{[\ell]}$ is given by
$$
\mathbb E_L|_{(I,J)} = \chi(\mathcal O, L) - \chi(J, I \otimes L).
$$
Let $\fL_m$ be the trivial line bundle on $X$ with a scaling action of $\C^*$ of 
character $z^m$, and let $\Delta_n$ be the diagonal in $\Xn \times \Xn$. Then, 
\begin{eqnarray}   \label{ResOfEToD}
\mathbb E_{\fL_m}|_{\Delta_n} = T_{\Xn, m},
\end{eqnarray}
the tangent bundle $T_\Xn$ with a scaling action of $\C^*$ of character $z^m$. 
By abusing notations, we also use $L$ to denote its first Chern class. Put
\begin{eqnarray}   
   \Gamma_{\pm}(z) 
&=&\exp \left ( \sum_{n>0} \frac{z^{\mp n}}{n} \fa_{\pm n} \right ), 
        \label{Gammaz}   \\
   \Gamma_{\pm}(L, z) 
&=&\exp \left ( \sum_{n>0} \frac{z^{\mp n}}{n} \fa_{\pm n}(L) \right ).   
        \label{GammaLz}
\end{eqnarray}

\begin{remark}  \label{SignDiff}
There is a sign difference between the Heisenberg commutation relations used in 
\cite{Car1} (see p.3 there) and in this paper (see Theorem~\ref{commutator}).
So for $n > 0$, our Heisenberg operators $\fa_{-n}(L)$ and $\fa_{n}(-L)$
are equal to the Heisenberg operators $\fa_{-n}(L)$ and $\fa_{n}(L)$ in \cite{Car1}.
Accordingly, our operators $\Gamma_-(L, z)$ and $\Gamma_+(-L, z)$ are equal to
the operators $\Gamma_-(L, z)$ and $\Gamma_+(L, z)$ in \cite{Car1}.
\end{remark}

The following commutation relations can be found in \cite{Car1} (see Remark~\ref{SignDiff}):
\begin{eqnarray}   
&[\Gamma_+(L, x), \Gamma_+(L', y)] = [\Gamma_-(L, x), \Gamma_-(L', y)] = 0,&
                                           \label{CarLemma5.1}          \\
&\Gamma_+(L, x)\Gamma_-(L', y) = (1-yx^{-1})^{\langle L, L' \rangle} \,
    \Gamma_-(L', y) \Gamma_+(L, x).&     \label{CarLemma5.2}
\end{eqnarray}

Let $\Wb(L, z): \fock \to \fock$ be the vertex operator constructed in \cite{CO, Car1} 
where $z$ is a formal variable. By \cite{Car1}, $\Wb(L, z)$ is defined via the pairing
\begin{eqnarray}   \label{def-WLz}
  \langle \Wb(L, z) \eta, \xi \rangle 
= z^{\ell-k} \cdot \int_{X^{[k]} \times X^{[\ell]}}
  (\eta \otimes \xi) \, c_{k+\ell}(\mathbb E_L)
\end{eqnarray}
for $\eta \in H^*(X^{[k]})$ and $\xi \in H^*(X^{[\ell]})$.
The main result in \cite{Car1} is (see Remark~\ref{SignDiff}):
\begin{eqnarray}   \label{WLz}
\Wb(L, z) = \Gamma_-(L-K_X, z) \, \Gamma_+(-L, z).
\end{eqnarray}

\begin{definition}
For $\alpha_1, \ldots, \alpha_N \in H^*(X)$ and 
integers $k_1, \ldots, k_N \ge 0$, define 
\begin{eqnarray}    \label{F-generating}
  F^{\alpha_1, \ldots, \alpha_N}_{k_1, \ldots, k_N}(q) 
= \sum_{n \ge 0} q^n \int_\Xn \left ( \prod_{i=1}^N 
  G_{k_i}(\alpha_i, n) \right ) c\big (T_\Xn \big )
\end{eqnarray}
where $c\big (T_\Xn \big )$ is the total Chern class of 
the tangent bundle $T_\Xn$ of $\Xn$. 
\end{definition}

By G\" ottsche's Theorem in \cite{Got}, we have
\begin{eqnarray}    \label{F-generating.1}
  F(q) 
= \sum_{n \ge 0} q^n \, \int_\Xn c(T_\Xn) 
= \sum_{n \ge 0} \chi(\Xn) q^n
= (q; q)_\infty^{-\chi(X)}. 
\end{eqnarray}
Inspired by \cite{Oko}, we define the reduced series
$$
(q; q)_\infty^{\chi(X)} \cdot F^{\alpha_1, \ldots, \alpha_N}_{k_1, \ldots, k_N}(q)
= \frac{F^{\alpha_1, \ldots, \alpha_N}_{k_1, \ldots, k_N}(q)}{F(q)}.
$$
The series $F^{\alpha_1, \ldots, \alpha_N}_{k_1, \ldots, k_N}(q)$
has been studied in \cite{QY}.
The following is \cite[Lemma~3.2]{QY}.

\begin{lemma}  \label{FtoW}
Let $\fn$ be the number-of-points operator, i.e., $\fn|_{H^*(\Xn)} = n \, {\rm Id}$. Then,
\begin{eqnarray}   \label{FtoW.0}
  F^{\alpha_1, \ldots, \alpha_N}_{k_1, \ldots, k_N}(q) 
= \Tr \, q^\fn \, \Wb(\fL_1, z) \, \prod_{i=1}^N \fG_{k_i}(\alpha_i)
\end{eqnarray}
where $\fL_1$ is the trivial line bundle on $X$ with a scaling action of 
$\C^*$ of character $z$.
\end{lemma}

\begin{remark}   \label{RMK_FtoW}
\begin{enumerate}
\item[{\rm (i)}]
In \eqref{FtoW.0}, we have implicitly set $t = 1$ for 
the generator $t$ of the equivariant cohomology $H^*_{\C^*}({\rm pt})$ 
of a point.

\item[{\rm (ii)}]
Let $1_X, K_X, e_X$ be the fundamental class, canonical divisor, 
Euler class of the surface $X$ respectively. 
By \eqref{FtoW.0} and Theorem~\ref{char_th}, 
$F^{\alpha_1, \ldots, \alpha_N}_{k_1, \ldots, k_N}(q)$ is 
an infinite linear combination of the expressions
\begin{eqnarray}   \label{epsilonialphai}
\Tr \, q^\fn \, {\bf W}(\fL_1, z) \, \prod_{i=1}^N 
     \frac{\fa_{\la^{(i)}}(\epsilon_i \alpha_i)}{\la^{(i)!}}
\end{eqnarray}
where $\epsilon_i \in \{1_X, e_X, K_X, K_X^2\}$ 
and $\la^{(1)}, \ldots, \la^{(N)}$ are generalized partitions satisfying 
$|\la^{(i)}| = 0$ and $\ell(\la^{(i)}) = k+2-|\epsilon_i|/2$.
\end{enumerate}
\end{remark}

Let $\lambda = \big (\cdots (-2)^{s_{-2}}(-1)^{s_{-1}} 
1^{s_1}2^{s_2} \cdots \big )$
and $\mu = \big (\cdots (-2)^{t_{-2}}(-1)^{t_{-1}} 
1^{t_1}2^{t_2} \cdots \big )$ be two generalized partitions. 
If $t_i \le s_i$ for every $i$, we write $\mu \le \la$ and define
$$
\la - \mu 
= \big (\cdots (-2)^{s_{-2} - t_{-2}}(-1)^{s_{-1}-t_{-1}} 1^{s_1-t_1}2^{s_2-t_2} \cdots \big ).
$$
The following is from the proof of \cite[Theorem~5.8]{Qin1} 
and deals with \eqref{epsilonialphai}. 

\begin{lemma}   \label{ProofOfQin5.8}
Let $\la^{(1)}, \ldots, \la^{(N)}$ be generalized partitions. Then, 
\begin{eqnarray}     \label{ThmJJkAlpha.3}
& &\Tr \, q^\fn \, {\bf W}(\fL_1, z) \, \prod_{i=1}^N 
     \frac{\fa_{\la^{(i)}}(\alpha_i)}{\la^{(i)!}}   \nonumber   \\
&=&z^{\sum_{i=1}^N |\la^{(i)}|} \cdot \sum_{\substack{\w \la^{(1)} 
     \le \la^{(1)}, \ldots, \w \la^{(N)} \le \la^{(N)}\\
     \sum_{i=1}^N (|\la^{(i)}| - |\w \la^{(i)}|) = 0}}
     \,\, \prod_{\substack{1 \le i \le N\\n \ge 1}} \left ( 
     \frac{(-1)^{p^{(i)}_{n}}}{p^{(i)}_{n}!} 
     \frac{q^{n p^{(i)}_{n}}}{(1-q^n)^{p^{(i)}_{n}}}  
     \frac{1}{\w p^{(i)}_{n}!} \frac{1}{(1-q^n)^{\w p^{(i)}_{n}}} \right )
     \nonumber   \\
& &\cdot \Tr \, q^\fn \prod_{i=1}^N  \frac{\fa_{\la^{(i)} - \w \la^{(i)}}
     \big ((1_X - K_X)^{\sum_{n \ge 1} p^{(i)}_{n}}\alpha_i \big )}
     {\big (\la^{(i)}-\w \la^{(i)} \big )^!}
\end{eqnarray}
where $\w \la^{(i)} = \big ( \cdots (-n)^{\w p^{(i)}_n} \cdots 
(-1)^{\w p^{(i)}_1} 1^{p^{(i)}_1} \cdots n^{p^{(i)}_n} \cdots \big ), 
1 \le i \le N$ are generalized partitions, together with the convention 
that for an empty generalized partition $\mu$,
\begin{eqnarray}     \label{IntBeta}
\frac{\fa_\mu(\beta)}{\mu^!} = \int_X \beta = \langle \beta, 1_X \rangle 
= \langle \beta \rangle.
\end{eqnarray}
\end{lemma}

The leading term in the series 
$F^{\alpha_1, \ldots, \alpha_N}_{k_1, \ldots, 
k_N}(q)$ follows from Remark~\ref{RMK_FtoW}~(ii) and the leading term 
on the the right-hand-side of \eqref{ThmJJkAlpha.3}. 
We refer to the Theorem~5.8 and Theorem~5.12 in \cite{Qin1} for details.
When $N = 1$, the right-hand-side of \eqref{ThmJJkAlpha.3} 
has been worked out explicitly in \cite[Remark~5.9]{Qin1} as follows. 

\begin{lemma}   \label{Qin5.9}
Let $\la = (\cdots (-n)^{\w m_n} \cdots 
(-1)^{\w m_1} 1^{m_1}\cdots n^{m_n} \cdots)$ be a generalized partition. 
For $n_1 \ge 1$ with $m_{n_1} \cdot \w m_{n_1} \ge 1$, define
$m_{n_1}(n_1) = m_{n_1}-1$, $\w m_{n_1}(n_1) = \w m_{n_1}-1$,
and $m_n(n_1) = m_n$, $\w m_n(n_1) = \w m_n$ if $n \ne n_1$. 
Then, 
$
\displaystyle{\Tr \, q^\fn \, {\bf W}(\fL_1, z) \, 
\frac{\fa_\la(\alpha)}{\la^!}}
$ 
equals
$$
{z^{|\la|}} {(q; q)_\infty^{-\chi(X)}} \cdot
   \langle (1_X - K_X)^{\sum_{n \ge 1} m_{n}}, \alpha \rangle   \cdot \prod_{n \ge 1} \left ( \frac{(-1)^{m_{n}}}{m_{n}!} \frac{q^{n m_n}}{(1-q^n)^{m_n}} 
   \frac{1}{\w m_{n}!} \frac{1}{(1-q^n)^{\w m_n}} \right )
$$
\begin{eqnarray*}    
&+ &{z^{|\la|}} {(q; q)_\infty^{-\chi(X)}} \cdot \langle e_X, \alpha \rangle 
   \cdot \sum_{n_1 \ge 1 \, \text{\rm with } m_{n_1} \cdot \w m_{n_1} \ge 1} 
   \frac{(-n_1) q^{n_1}}{1-q^{n_1}} \cdot                  \nonumber   \\
& &\cdot \prod_{n \ge 1} \left ( \frac{(-1)^{m_{n}(n_1)}}{m_{n}(n_1)!} 
   \frac{q^{n m_n(n_1)}}{(1-q^n)^{m_n(n_1)}} 
   \frac{1}{\w m_{n}(n_1)!} \frac{1}{(1-q^n)^{\w m_n(n_1)}} \right ).              
\end{eqnarray*}
\end{lemma}

The next lemma deals with the trace on line \eqref{ThmJJkAlpha.3}.

\begin{lemma}  \label{TrqnfamuiNot0}
For $1 \le i \le N$, let $\ell(\mu^{(i)})$ be non-empty generalized 
partitions and $\beta_i$ be homogeneous cohomology class. If 
\begin{eqnarray}     \label{TrNot0Begin}
\Tr \, q^\fn \prod_{i=1}^N \frac{\fa_{\mu^{(i)}}(\beta_i)}{\mu^{(i)!}} \ne 0,
\end{eqnarray}
then all of the following must be satisfied:
\begin{enumerate}
\item[(i)] $\sum_{i=1}^N \big (|\mu^{(i)}|, 2(\ell(\mu^{(i)}) - 2) + 
|\beta_i| \big ) = 0$.

\item[(ii)] Let $n \in \Z$. Then the multiplicities of $n$ and $-n$ 
as parts of the generalized partition $\sum_{i=1}^N \mu^{(i)}$ 
are the same.
\end{enumerate}
\end{lemma}
\begin{proof}
By \cite[(5.18)]{Qin1}, the bi-degree of $\fa_\mu(\beta)$ is
\begin{eqnarray}     \label{20161207.1}
\Big (|\mu|, 2(\ell(\mu) - 2) + |\beta| \Big ).
\end{eqnarray}
So the bi-degree of $\prod_{i=1}^N \fa_{\mu^{(i)}}(\beta_i)$ is 
equal to 
$$
\sum_{i=1}^N \big (|\mu^{(i)}|, 2(\ell(\mu^{(i)}) - 2) + 
|\beta_i| \big ),
$$
and (i) follows. To prove (ii), let $n$ be a part of 
$\sum_{i=1}^N \mu^{(i)}$. Let the multiplicities of $n$ and $-n$ 
as parts of $\sum_{i=1}^N \mu^{(i)}$ be $k_1$ and $k_2$ respectively. 
Denote the left-hand-side of \eqref{TrNot0Begin} by ${\rm Tr}$. 
Then, ${\rm Tr}$ is a finite sum of expressions of the form
$$
\Tr \, q^\fn \fa_{\pm n}(\alpha) \prod_{j=1}^m \fa_{n_j}(\alpha_j)
$$
where $1+m = \sum_{i=1}^N \ell(\mu^{(i)})$, 
$(\pm n) + \sum_{j=1}^m n_j = 0$, and $n$ and $-n$ appear $k_1$ and $k_2$ times respectively in the list $\pm n, n_1, \ldots n_m$. 
Assume $k_1 \ne k_2$.
To get a contradiction to \eqref{TrNot0Begin}, it sufficies to show that 
\begin{eqnarray}     \label{TrqnfamuiNot0.1}
\Tr \, q^\fn \fa_n(\alpha) \prod_{j=1}^m \fa_{n_j}(\alpha_j) = 0.
\end{eqnarray}

First of all, assume that $n < 0$. Then, 
\begin{eqnarray*}     
& &\Tr \, q^\fn \fa_n(\alpha) \prod_{j=1}^m \fa_{n_j}(\alpha_j)     \\
&=&q^{-n} \Tr \, \fa_n(\alpha) q^\fn \prod_{j=1}^m \fa_{n_j}(\alpha_j) \\
&=&q^{-n} \Tr \, q^\fn \prod_{j=1}^m \fa_{n_j}(\alpha_j) 
     \cdot \fa_n(\alpha) \\
&=&q^{-n} \Tr \, q^\fn \fa_n(\alpha) \prod_{j=1}^m \fa_{n_j}(\alpha_j)
   + n q^{-n} \sum_{n_{j} = -n} \langle \alpha_j, \alpha \rangle 
   \cdot \Tr \, q^\fn 
   \prod_{1 \le j_1 \le m, j_1 \ne j} \fa_{n_{j_1}}(\alpha_{j_1}). 
\end{eqnarray*}
It follows that the left-hand-side of \eqref{TrqnfamuiNot0.1} is equal to 
\begin{eqnarray}     \label{TrqnfamuiNot0.2}
\frac{n q^{-n}}{1 - q^{-n}} \sum_{n_{j} = -n} \langle \alpha_j, \alpha \rangle \cdot \Tr \, q^\fn 
   \prod_{1 \le j_1 \le m, j_1 \ne j} \fa_{n_{j_1}}(\alpha_{j_1}).
\end{eqnarray}

Next, assume that $n > 0$. Then, 
$$
\Tr \, q^\fn \fa_n(\alpha) \prod_{j=1}^m \fa_{n_j}(\alpha_j)     
= \Tr \, \prod_{j=1}^m \fa_{n_j}(\alpha_j) 
     \cdot q^\fn \fa_n(\alpha) 
= q^{-n} \Tr \, q^\fn \prod_{j=1}^m \fa_{n_j}(\alpha_j) 
     \cdot \fa_n(\alpha)
$$
since $\sum_{j=1}^m n_j = -n < 0$. Therefore, 
\begin{eqnarray*}     
& &\Tr \, q^\fn \fa_n(\alpha) \prod_{j=1}^m \fa_{n_j}(\alpha_j)  \\   
&=&q^{-n} \Tr \, q^\fn \fa_n(\alpha) \prod_{j=1}^m \fa_{n_j}(\alpha_j)
   + n q^{-n} \sum_{n_{j} = -n} \langle \alpha_j, \alpha \rangle 
   \cdot \Tr \, q^\fn 
   \prod_{1 \le j_1 \le m, j_1 \ne j} \fa_{n_{j_1}}(\alpha_{j_1})
\end{eqnarray*}
It follows that the left-hand-side of \eqref{TrqnfamuiNot0.1} is equal to 
\begin{eqnarray}     \label{TrqnfamuiNot0.3}
\frac{-n}{1 - q^{n}} \sum_{n_{j} = -n} \langle \alpha_j, \alpha \rangle 
   \cdot \Tr \, q^\fn 
   \prod_{1 \le j_1 \le m, j_1 \ne j} \fa_{n_{j_1}}(\alpha_{j_1}).
\end{eqnarray}

Note that in both \eqref{TrqnfamuiNot0.2} and \eqref{TrqnfamuiNot0.3},
$n$ and $-n$ appear $k_1-1$ and $k_2-1$ times respectively in
every list $n_{j_1}$ with $1 \le j_1 \le m$ and $j_1 \ne j$. 
Therefore, \eqref{TrqnfamuiNot0.1} follows from an induction on 
$k_1$ and $k_2$. 
\end{proof}

Assume that $\sum_{i=1}^N |\mu^{(i)}| = 0$ and $|\mu^{(i_0)}| < 0$ 
for some $1 \le i_0 \le N$. Put $n_0 = -|\mu^{(i_0)}|$. Then, $n_0 > 0$. 
By \cite[(5.21)]{Qin1}, 
$\displaystyle{\Tr \, q^\fn \prod_{i=1}^N \frac{\fa_{\mu^{(i)}}(\beta_i)}
  {\mu^{(i)!}}}$ is equal to 
\begin{eqnarray}  
& &\frac{q^{n_0}}{1 - q^{n_0}} \, \sum_{r = i_0+1}^N \Tr \, q^\fn 
     \prod_{1 \le i \le r-1, i \ne i_0} \frac{\fa_{\mu^{(i)}}(\beta_i)}{\mu^{(i)!}}   \nonumber  \\
& &\quad \cdot \left [\frac{\fa_{\mu^{(r)}}(\beta_r)}{\mu^{(r)!}},
      \frac{\fa_{\mu^{(i_0)}}(\beta_{i_0})}{\mu^{({i_0})!}} \right ] \cdot
      \prod_{i = r+1}^N \frac{\fa_{\mu^{(i)}}(\beta_i)}{\mu^{(i)!}}   \nonumber  \\
&+&\frac{1}{1 - q^{n_0}} \, 
    \sum_{r = 1}^{i_0 -1} \Tr \, q^\fn \prod_{i=1}^{r-1} \frac{\fa_{\mu^{(i)}}(\beta_i)}{\mu^{(i)!}}
    \cdot \left [\frac{\fa_{\mu^{(r)}}(\beta_r)}{\mu^{(r)!}},
      \frac{\fa_{\mu^{(i_0)}}(\beta_{i_0})}{\mu^{({i_0})!}} \right ]     \nonumber \\
& &\quad \cdot \prod_{r+1 \le i \le N, i \ne i_0} \frac{\fa_{\mu^{(i)}}(\beta_i)}{\mu^{(i)!}}.     \label{Trqdjpx.3} 
\end{eqnarray}

\section{\bf $\langle \ch_1 \ch_1 \rangle'$ 
in the equivariant setting and quasi-modular forms} 
\label{sect_CompEquiv}

In order to compute the reduced generating series 
$\langle \ch_1^{L_1} \ch_1^{L_2} \rangle'$ where $L_1$ and $L_2$ are 
line bundles on a smooth projective surface $X$, we have to understand 
certain quasi-modular forms which appear in a new reduced generating series 
$\langle \ch_1 \ch_1 \rangle'$. The series $\langle \ch_1 \ch_1 \rangle'$ 
is the counterpart of $\langle \ch_1^{L_1} \ch_1^{L_2} \rangle'$ 
in the equivariant setting.
In this section, we will recall the main results of \cite{Car1, Car2, SQ} 
which deal with the equivariant setting. 
We will follow the setups in \cite{SQ} and compute 
$\langle \ch_1 \ch_1 \rangle'$. In view of \eqref{Z246}, we will express 
the quasi-modular forms appearing in $\langle \ch_1 \ch_1 \rangle'$ 
by using $Z(2), Z(4)$ and $Z(6)$. 

Throughout this section, let $X = \C^2$. 
Let $u, v$ be the standard coordinate functions on $\mathbb C^2$.
Define the action of $\T = \C^*$ on $X$ by
\begin{eqnarray}  \label{action_plane}
s\cdot (u, v)=(su, s^{-1}v), \quad s\in \T.
\end{eqnarray}
The origin of $X$ is the only fixed point. This action of $\T$ on $X$ induces 
natural actions of $\T$ on the Hilbert schemes $\Xn$. Define a bilinear form 
$$
\langle \cdot, \cdot \rangle: \,\, \big (H^*_{\mathbb T}(\Xn) \otimes_{\C[t]} \C(t) \big )
\times \big (H^*_{\mathbb T}(\Xn) \otimes_{\C[t]} \C(t) \big ) \to \mathbb C(t)
$$ 
by putting
%(also cf. \cite{Vas})
\begin{eqnarray} \label{bilinear1}
\langle A, B \rangle =(-1)^n p_!\iota_!^{-1}(A\cup B)
\end{eqnarray}
where $p$ is the projection $(\Xn)^{\mathbb T}\to \text{pt}$, 
and $\iota: (\Xn)^{\mathbb T}\to \Xn$ is the inclusion map. 
%This induces a bilinear form $\langle \cdot, \cdot \rangle$ on $\mathbb H_X'$.

%In this equivariant setting,
A Heisenberg algebra action on the middle cohomology groups
\begin{eqnarray*}   
\fock = \bigoplus_{n \ge 0} H_\T^{2n}(\Xn)
\end{eqnarray*}
was constructed in \cite{Vas} (see also \cite{Gro, LQW3, Nak}). 
Equivalently, the space
\begin{eqnarray}  \label{def-fockprime}
\fockprime 
= \bigoplus_{n \ge 0} H_\T^*(\Xn) \otimes_{\C[t]} \C(t)
\end{eqnarray}
is an irreducible module over the Heisenberg algebra generated by 
the Nakajima operators $\fa_m(\alpha), m \in \Z$ and $\alpha \in 
H_\T^*(X) \otimes_{\C[t]} \C(t) = \C(t) \cdot 1_X$ with 
\begin{eqnarray}  \label{comm_rel}
[\fa_m(\alpha), \fa_n(\beta)]
= m \cdot \delta_{m, -n} \cdot \langle \alpha, \beta \rangle
\end{eqnarray}
where $\langle \alpha, \beta \rangle$ denotes the equivariant pairing \eqref{bilinear1}.
%By the localization theorem, 
%$$
%1_X = -t^{-2} [O]
%$$ 
%where $O$ denotes the origin of $X = \C^2$. By \eqref{bilinear1},
%\begin{eqnarray}  \label{201706231043am}
%\langle 1_X, 1_X \rangle = t^{-2}.
%\end{eqnarray}
For every integer $m$, put
$$
\fa_m = \fa_m(1_X),
$$
and $\vac = 1 \in H_\T^*(X^{[0]}) \otimes_{\C[t]} \C(t) = \C(t)$.
Since $H_\T^*(X) \otimes_{\C[t]} \C(t) = \C(t) \cdot 1_X$, 
\begin{eqnarray}  \label{fockprime}
\fockprime = \C(t)[\fa_{-1}, \fa_{-2}, \ldots ] \cdot \vac
\end{eqnarray}
together with the Heisenberg commutation relation
\begin{eqnarray}  \label{comm_rel_1X}
[\fa_m, \fa_n] = m \cdot \delta_{m, -n} \cdot t^{-2}.
\end{eqnarray}

In the rest of the paper, we set $t = 1$. We see from \eqref{fockprime}
and \eqref{comm_rel_1X} that 
\begin{eqnarray}  \label{fockprimeT=1}
\fockprime = \C[\fa_{-1}, \fa_{-2}, \ldots ] \cdot \vac
\end{eqnarray}
together with the following Heisenberg commutation relation
\begin{eqnarray}  \label{comm_rel_1XT=1}
[\fa_m, \fa_n] = m \cdot \delta_{m, -n}.
\end{eqnarray}

\begin{remark}   \label{RmkSigns}
In our setup here, we have implicitly defined 
\begin{eqnarray}  \label{RmkSigns.1}
\fa_n = (-1)^n \cdot \fa_{-n}^*
%\fa_n(\alpha) = (-1)^n \cdot \big (\fa_{-n}(\alpha) \big )^*
\end{eqnarray}
for $n > 0$, which is consistent with the setup in \cite{LQW2}.
The commutation relation \eqref{comm_rel_1XT=1} is 
the same as the commutation relation in \cite[(27)]{Car2}. 
\end{remark}

Recall from \eqref{L[n]} the tautological rank-$n$ vector bundle 
$\On$ over $\Xn$, which is $\T$-equivariant. 
Let $\ch_{k, \T} \big ( \On \big )$ 
be its $k$-th $\T$-equivariant Chern character.

\begin{definition}   \label{def-Gk}
For $k \ge 0$, define $\fG_k$ to be {\it the $k$-th equivariant Chern character operator} 
which acts on $\fockprime$ by cup product with 
$\bigoplus_n \ch_{k, \T} \big ( \On \big )$.
\end{definition}

For a generalized partition $\la = \big (\cdots (-2)^{m_{-2}}(-1)^{m_{-1}} 
1^{m_1}2^{m_2} \cdots \big )$, put
\begin{eqnarray*}
\fa_{\lambda} = \prod_i \fa_i^{m_i}
= \big (\cdots \fa_{-2}^{m_{-2}} \fa_{-1}^{m_{-1}} \fa_{1}^{m_{1}} \fa_{2}^{m_{2}}\cdots \big ).
\end{eqnarray*}
%where $\prod_i \fa_i^{m_i} $ is understood to be
%$\cdots \fa_{-2}^{m_{-2}} \fa_{-1}^{m_{-1}} \fa_{1}^{m_{1}} 
%\fa_{2}^{m_{2}}\cdots$.
The following is from the Proposition~3.3 and Corollary~3.4 of \cite{SQ}.

\begin{proposition}   \label{str-Gk} 
The equivariant Chern character operator $\fG_k$ is equal to
$$
\sum_{\ell(\lambda) \le k+2, |\lambda|=0} \frac{\fa_{\lambda}}{\lambda^!} \cdot 
\Coe_{z^k} \frac{1}{(x-1)(1-x^{-1})} 
\prod_{n > 0} \left ( \frac{x^n-1}{n} \right )^{m_{-n}}
\cdot \prod_{n > 0} \left ( \frac{1-x^{-n}}{n} \right )^{m_n}
$$
where $|x^{-1}| < 1$, $x = e^z$, and $\la = \big (\cdots (-2)^{m_{-2}}
(-1)^{m_{-1}} 1^{m_1}2^{m_2} \cdots \big )$. In particular,
\beq   \label{str-Gk.0}
\fG_1 = \sum_{\ell(\la) = 3, |\la|=0} \frac{\fa_{\la}}{\la^!}.
\eeq
\end{proposition}

Fix $m \in \Z$. Let 
\beq   \label{FtoG}
\langle \ch_{k_1} \cdots \ch_{k_N} \rangle 
= \Tr \, q^\fn \, \Gamma_-(z)^m \Gamma_+(z)^{-m} \, \prod_{i=1}^N \fG_{k_i}
\eeq
where $\fn$ is the number-of-points operator (i.e., 
$\fn|_{H_\T^*(\Xn)} = n \, {\rm Id}$) and 
$\Gamma_\pm(z)$ is from \eqref{Gammaz}.
By the result in Subsection~4.1 of \cite{Car1}, 
\beq   \label{Fmq}
\langle \rangle = (q;q)_\infty^{m^2-1}.
\eeq
Following \cite{Oko}, we define the reduced series 
\beq   \label{ReducedFk1N}
  \langle \ch_{k_1} \cdots \ch_{k_N} \rangle'
= \frac{\langle \ch_{k_1} \cdots \ch_{k_N} \rangle}{\langle \rangle}
= (q;q)_\infty^{-m^2+1} \cdot \langle \ch_{k_1} \cdots \ch_{k_N} \rangle.
\eeq
Parts (i) and (ii) of the following lemma are proved in 
\cite{Car1, Car2} and \cite{SQ} respectively.

\begin{lemma}   \label{kOdd}
Let $k_1, \ldots, k_N \ge 0$, and $m \in \Z$. 
\begin{enumerate}
\item[{\rm (i)}] As a function of $q$, the reduced series 
$\langle \ch_{k_1} \cdots \ch_{k_N} \rangle'$ 
is a quasi-modular form of weight at most $\sum_{i=1}^N (k_i + 2)$. 
As a function of $m$, it is a polynomial in $m^2$ of degree at most
$\sum_{i=1}^N \big (\lfloor \frac{k_i}{2} \rfloor + 1 \big )$.

\item[{\rm (ii)}] If $\sum_{i=1}^N k_i$ is odd, 
then $\langle \ch_{k_1} \cdots \ch_{k_N} \rangle' = 0$.
\end{enumerate}
\end{lemma}

In particular, as a function of $q$, $\langle \ch_1 \ch_1 \rangle'$ 
is a quasi-modular form of weight at most $6$. As a function of $m$, 
it is a polynomial in $m^2$ of degree at most $2$. 
In the rest of this section, we will compute 
$\langle \ch_1 \ch_1 \rangle'$ explicitly 
via the approaches established in \cite{SQ}. 
We begin with a lemma computing various traces in equivariant setting.

\begin{lemma}   \label{Trala}
Let $i$ and $j$ be positive integers. Then,
\begin{eqnarray}   
   \Tr \, q^\fn \fa_{-i}\fa_i 
&=&(q;q)_\infty^{-1} \cdot \frac{iq^i}{1-q^i},   \label{Trala.01}   \\
   \Tr \, q^\fn \fa_i \fa_{-i}
&=&(q;q)_\infty^{-1} \cdot \frac{i}{1-q^i},   \label{Trala.02}    \\
   \Tr \, q^\fn \fa_i\fa_j \fa_{-i}\fa_{-j} 
&=&(q;q)_\infty^{-1} \cdot 
     \frac{(1+\delta_{i,j})ij}{(1 - q^{i})(1 - q^{j})}, \label{Trala.05}\\
   \Tr \, q^\fn \fa_{-i}\fa_{-j}\fa_i\fa_j 
&=&(q;q)_\infty^{-1} \cdot   
     \frac{(1+\delta_{i,j})ijq^{i+j}}{(1 - q^{i})(1 - q^{j})},  
     \label{Trala.06}    \\
   \Tr \, q^\fn \fa_{-i} \fa_{j} \fa_{-j}\fa_{i} 
&=&(q;q)_\infty^{-1} \cdot \frac{ijq^i(1 + \delta_{i,j} q^j)}
    {(1-q^i)(1-q^j)},  \label{Trala.07}   \\
   \Tr \, q^\fn \fa_{-i} \fa_{-j} \fa_{i+j} \fa_{-i-j}\fa_{i} \fa_{j}
&=&(q;q)_\infty^{-1} \cdot 
   \frac{(1+\delta_{i,j})ij(i+j)q^{i+j}}{(1-q^i)(1-q^j)(1-q^{i+j})},  
     \label{Trala.08}   \\
   \Tr \, q^\fn \fa_{-i-j} \fa_{i} \fa_{j} \fa_{-i}\fa_{-j} \fa_{i+j}
&=&(q;q)_\infty^{-1} \cdot \frac{}{}
  \frac{(1+\delta_{i,j})ij(i+j)q^{i+j}}{(1-q^i)(1 - q^j)(1-q^{i+j})}.
     \label{Trala.09} 
\end{eqnarray}
\end{lemma}
\begin{proof}
Recall the following commutation relation \eqref{comm_rel_1XT=1}. 
Since $\Tr \, q^\fn = (q;q)_\infty^{-1}$, 
$$
   \Tr \, q^\fn \fa_{-i}\fa_i 
= q^i \Tr \, \fa_{-i} q^\fn \fa_i = q^i \Tr \, q^\fn \fa_i \fa_{-i} 
= q^i \Tr \, q^\fn \fa_{-i} \fa_i + iq^i \cdot (q;q)_\infty^{-1}.
$$
So \eqref{Trala.01} follows. Moreover, 
$\Tr \, q^\fn \fa_i \fa_{-i}
= \Tr \, q^\fn \fa_{-i} \fa_i + i \cdot (q;q)_\infty^{-1}$.
Thus \eqref{Trala.02} follows from \eqref{Trala.01}.

Similarly, $\Tr \, q^\fn \fa_i\fa_j \fa_{-i}\fa_{-j}$ is equal to
\begin{eqnarray*}
& &\Tr \, q^\fn \fa_{-i} \fa_{-j} \fa_i\fa_j 
     + (1+\delta_{i,j}) i \, \Tr \, q^\fn \fa_j \fa_{-j} 
     + (1+\delta_{i,j}) j \, \Tr \, q^\fn \fa_{-i} \fa_i   \\
&=&q^{i+j} \Tr \, \fa_{-i} \fa_{-j} q^\fn \fa_i\fa_j 
     + (1+\delta_{i,j}) i \, \frac{(q;q)_\infty^{-1} \cdot j}{1-q^j} 
     + (1+\delta_{i,j}) j \, \frac{(q;q)_\infty^{-1} \cdot iq^i}{1-q^i} \\
&=&q^{i+j} \Tr \, q^\fn \fa_i\fa_j \fa_{-i} \fa_{-j} 
     + (q;q)_\infty^{-1} \cdot 
     \frac{(1+\delta_{i,j}) ij(1-q^{i+j})}{(1-q^j)(1-q^j)}.
\end{eqnarray*}
So \eqref{Trala.05} follows. Moreover, from the above argument, 
$\Tr \, q^\fn \fa_{-i} \fa_{-j} \fa_i\fa_j$ is equal to
$$
\Tr \, q^\fn \fa_i\fa_j \fa_{-i}\fa_{-j}
- (1+\delta_{i,j}) i \, \Tr \, q^\fn \fa_j \fa_{-j} 
     - (1+\delta_{i,j}) j \, \Tr \, q^\fn \fa_{-i} \fa_i.
$$
Therefore, \eqref{Trala.06} follows from \eqref{Trala.01}, 
\eqref{Trala.02} and \eqref{Trala.05}. Since 
$$
  \Tr \, q^\fn \fa_{-i} \fa_j \fa_{-j}\fa_{i}
= \Tr \, q^\fn \fa_{-i} \fa_{-j} \fa_j \fa_{i} 
  + j \Tr \, q^\fn \fa_{-i} \fa_{i},
$$
\eqref{Trala.06} follows from \eqref{Trala.01}, 
\eqref{Trala.02} and \eqref{Trala.05}.

Next, $\Tr \, q^\fn \fa_{-i} \fa_{-j} \fa_{i+j} \fa_{-i-j}\fa_{i} \fa_{j}$ 
is equal to 
\begin{eqnarray*}
& &q^i \Tr \, \fa_{-i} q^\fn \fa_{-j}\fa_{i+j}\fa_{-i-j}\fa_{i}\fa_{j} \\
&=&q^i \Tr \, q^\fn \fa_{-j} \fa_{i+j}\fa_{-i-j}\fa_{i}\fa_{j}\fa_{-i} \\
&=&q^i \Tr \, q^\fn \fa_{-i}\fa_{-j}\fa_{i+j}\fa_{-i-j}\fa_{i}\fa_{j} 
  + (1+\delta_{i,j})iq^i \Tr \, q^\fn \fa_{-j}\fa_{i+j}\fa_{-i-j}\fa_{j}.
\end{eqnarray*}
It follows that 
$\Tr \, q^\fn \fa_{-i} \fa_{-j} \fa_{i+j} \fa_{-i-j}\fa_{i} \fa_{j}$ 
is equal to 
\begin{eqnarray*}
  \frac{(1+\delta_{i,j})iq^i}{1-q^i} 
   \Tr \, q^\fn \fa_{-j}\fa_{i+j}\fa_{-i-j}\fa_{j}
&=&\frac{(1+\delta_{i,j})iq^i}{1-q^i} \cdot (q;q)_\infty^{-1} \cdot  
   \frac{j(i+j)q^j}{(1-q^j)(1-q^{i+j})}  \\
&=&(q;q)_\infty^{-1} \cdot 
   \frac{(1+\delta_{i,j})ij(i+j)q^{i+j}}{(1-q^i)(1-q^j)(1-q^{i+j})}.
\end{eqnarray*}
This verifies \eqref{Trala.08}. Finally, 
$\Tr \, q^\fn \fa_{-i-j} \fa_{i} \fa_{j} \fa_{-i}\fa_{-j} \fa_{i+j}$ 
is equal to
\begin{eqnarray*}
& &\Tr \, q^\fn \fa_{-i-j} \fa_{i} \fa_{j} \fa_{-i}\fa_{-j} \fa_{i+j}  \\
&=&q^{i+j} \Tr \, \fa_{-i-j} q^\fn 
   \fa_{i} \fa_{j} \fa_{-i}\fa_{-j} \fa_{i+j}  \\
&=&q^{i+j} \Tr \, q^\fn 
   \fa_{i} \fa_{j} \fa_{-i}\fa_{-j} \fa_{i+j} \fa_{-i-j}  \\
&=&q^{i+j} \Tr \, q^\fn \fa_{-i-j}
   \fa_{i} \fa_{j} \fa_{-i}\fa_{-j} \fa_{i+j} 
   + (i+j) q^{i+j} \Tr \, q^\fn \fa_{i} \fa_{j} \fa_{-i}\fa_{-j}.
\end{eqnarray*}
Hence $\Tr \, q^\fn \fa_{-i-j} \fa_{i} \fa_{j} \fa_{-i}\fa_{-j} \fa_{i+j}$ 
is equal to
$$
  \frac{(i+j) q^{i+j}}{1-q^{i+j}} 
  \Tr \, q^\fn \fa_{i} \fa_{j} \fa_{-i}\fa_{-j}
= (q;q)_\infty^{-1} \cdot \frac{}{}
  \frac{(1+\delta_{i,j})ij(i+j)q^{i+j}}{(1-q^i)(1 - q^j)(1-q^{i+j})}
$$
where we have used \eqref{Trala.05} in the last step.
This confirms \eqref{Trala.09}.
\end{proof}

\begin{definition}   \label{h11}
We define the following three functions of $q$:
\begin{eqnarray}
   h_{1,1}^{(0)}(q) 
&=&\sum_{i, j > 0} \frac{ij(i+j)q^{i+j}}{(1-q^i)(1 - q^j)(1-q^{i+j})},
             \label{h11.01}    \\ 
   h_{1,1}^{(2)}(q)
&=&-\sum_{i,j > 0} \frac{j (i+j)(q^{i+j}+q^{2i+j})}
     {(1-q^i)^2(1-q^j)(1-q^{i+j})}   \nonumber    \\
& &  \quad - \frac12 \sum_{i, j > 0} \frac{ij(q^{i+j}+q^{2i+2j})}
     {(1 - q^{i})(1 - q^{j})(1-q^{i+j})^2},   \label{h11.02}   \\ 
   h_{1,1}^{(4)}(q)
&=&\frac14 \sum_{i,j,k, \ell >0, i+j=k+\ell} \frac{(i+j)q^{i+j}
     (1+q^{i+j})}{(1-q^i)(1-q^j)(1-q^k)(1-q^{\ell})(1-q^{i+j})}  \nonumber \\
& &  \quad - \sum_{i,j,k >0} \frac{(i+j)q^{i+j+k}(1+q^{i+j})}
     {(1-q^i)(1-q^j)(1-q^k)(1-q^{i+j})(1-q^{i+j+k})}  \nonumber \\
& &  \quad + \sum_{i,j,k >0} \frac{kq^{i+j+k}(1+q^{k})}
     {(1-q^i)(1-q^j)(1-q^k)(1-q^{i+k})(1-q^{j+k})}.   \label{h11.03}
\end{eqnarray}
\end{definition}

\begin{lemma}   \label{EquivCh1Ch1Prop}
Let $h_{1,1}^{(0)}(q), h_{1,1}^{(2)}(q), h_{1,1}^{(4)}(q)$
be from Definition~\ref{h11}. Then, 
\beq   \label{EquivCh1Ch1Prop.0}
  \langle \ch_1 \ch_1 \rangle'
= h_{1,1}^{(4)}(q) \cdot m^4 + h_{1,1}^{(2)}(q) \cdot m^2 
  + h_{1,1}^{(0)}(q).
\eeq
In particular, $h_{1,1}^{(4)}(q), h_{1,1}^{(2)}(q)$ and $h_{1,1}^{(0)}$ 
are quasi-modular forms of weight at most $6$.
\end{lemma}
\begin{proof}
The second conclusion follows from Lemma~\ref{kOdd}~(i). 
In the following, we will prove \eqref{EquivCh1Ch1Prop.0}.
By \eqref{ReducedFk1N}, \eqref{FtoG} and \eqref{str-Gk.0}, 
$\langle \ch_1 \ch_1 \rangle'$ is equal to 
\beq   \label{EquivCh1Ch1}
(q;q)_\infty^{-m^2+1} \cdot 
\sum_{\ell(\la) = \ell(\mu) = 3, |\la|= |\mu|=0} 
\Tr \, q^\fn \, \Gamma_-(z)^m \Gamma_+(z)^{-m} \, 
\frac{\fa_{\la}}{\la^!} \frac{\fa_{\mu}}{\mu^!}.
\eeq
For fixed generalized partitions $\la$ and $\mu$ in \eqref{EquivCh1Ch1}, 
denote the trace in \eqref{EquivCh1Ch1} by $\Tr_{\la, \mu}$. 
By \cite[(4.13)]{SQ}, $\Tr_{\la, \mu}$ is equal to
$$
(q; q)_\infty^{m^2} \cdot 
\sum_{\substack{\w \la^{(1)} \le \la^{(1)}, \w \la^{(2)} \le \la^{(2)}\\
     |\w \la^{(1)}| + |\w \la^{(2)}|= 0}} \, 
m^{\ell(\w \la^{(1)}) + \ell(\w \la^{(2)})} \cdot 
\Tr \, q^\fn \prod_{i=1}^2 \frac{\fa_{\la^{(i)}-\w \la^{(i)}}}
     {\big (\la^{(i)}-\w \la^{(i)}\big )^!}
$$
\begin{eqnarray}    \label{TrLa1Lak.2}
\cdot \prod_{1 \le i \le 2, n \ge 1} \left ( \frac{1}{p^{(i)}_{n}!} 
    \frac{q^{n p^{(i)}_{n}}}{(1-q^n)^{p^{(i)}_{n}}} 
\cdot \frac{(-1)^{\w p_n^{(i)}}}{\w p^{(i)}_{n}!} 
    \frac{1}{(1-q^n)^{\w p^{(i)}_{n}}} \right )
\end{eqnarray}
where $\la^{(1)} = \la, \la^{(2)} = \mu$, and $p^{(i)}_{n}$ and 
$\w p^{(i)}_{n}$ are the multiplicities of the parts $n > 0$ and $-n$ 
in $\w \la^{(i)}$ respectively. From the proof of \cite[Lemma~4.7]{SQ}, 
we see that the sum of all the contributions of the cases 
$\w \la^{(1)} = \la^{(1)}$ or $\w \la^{(2)} = \la^{(2)}$ in 
\eqref{TrLa1Lak.2} to \eqref{EquivCh1Ch1} is equal to zero. 
So the degree of $\langle \ch_1 \ch_1 \rangle'$ as a polynomial of $m$ 
is at most $4$. In fact, setting $M=N=2$ and $k_1=k_2=1$ in \cite[Corollary~4.9~(ii)]{SQ}, we conclude that the coefficient of $m^4$ 
in $\langle \ch_1 \ch_1 \rangle'$ is equal to
$$
\sum_{n \ne 0} \Coe_{z_{1}^0z_2^0} \left (
\Theta_0(q, z_1) \cdot \Theta_0(q, z_2) \cdot 
\frac{(-n)q^{-n}}{1-q^{-n}} (z_1z_2^{-1})^{n} \right ).
$$
where 
\beq     \label{i+j=k+l.1}
  \Theta_k(q, z)
= \frac{1}{(k+2)!} \cdot \left (\sum_{m > 0} \frac{(q z)^{m}}{1-q^{m}}
- \sum_{m > 0} \frac{z^{-m}}{1-q^{m}} \right )^{k+2}.
\eeq
A straightforward computation shows that the coefficient of $m^4$ 
in $\langle \ch_1 \ch_1 \rangle'$ is precisely $h_{1,1}^{(4)}(q)$.

Next, by \cite[Lemma~4.4~(i)]{SQ}, if 
$$
\Tr \, q^\fn \prod_{i=1}^2 
\frac{\fa_{\mu^{(i)}}}{\big (\mu^{(1)}\big )^!} \ne 0,
$$
then for any positive integer $n$, the multiplicities of the parts 
$n$ and $-n$ in $\mu^{(1)} + \mu^{(2)}$ must be the same. 
Thus, by \eqref{EquivCh1Ch1} and \eqref{TrLa1Lak.2},
$\langle \ch_1 \ch_1 \rangle'$ is equal to 
\beq   \label{EquivCh1Ch1.simplify1}
h_{1,1}^{(4)}(q) \cdot m^4 + (q; q)_\infty \cdot 
\sum_{\ell(\la) = 3, |\la| =0}
\Tr \, q^\fn \, \frac{\fa_{\la}}{\la^!} \frac{\fa_{-\la}}{(-\la)^!}
\eeq
\beq   \label{EquivCh1Ch1.simplify2}
+ m^2 \cdot (q;q)_\infty \cdot \sum_{\ell(\la) = 3, |\la| =0} \,
\sum_{(k) < \la} \, \frac{-q^{|k|}}{(1-q^{|k|})^2} 
\Tr \, q^\fn \frac{\fa_{\la-(k)}}{(\la-(k))^!} 
\frac{\fa_{(-\la)-(-k)}}{((-\la)-(-k))^!}.
\eeq

For line \eqref{EquivCh1Ch1.simplify1}, $\la = (-i, -j, i+j)$ or 
$(-i-j, i, j)$ for some positive integers $i$ and $j$.
Thus, by \eqref{Trala.08} and \eqref{Trala.09}, 
line \eqref{EquivCh1Ch1.simplify1} is equal to
\beq   \label{EquivCh1Ch1.simplify3}
  h_{1,1}^{(4)}(q) \cdot m^4 + 
  \sum_{i, j > 0} \frac{ij(i+j)q^{i+j}}{(1-q^i)(1 - q^j)(1-q^{i+j})}
= h_{1,1}^{(4)}(q) \cdot m^4 + h_{1,1}^{(0)}(q).
\eeq

For line \eqref{EquivCh1Ch1.simplify2}, if $k > 0$, then we have 
the following 2 cases:
\begin{enumerate}
\item[$\bullet$] $\la = (-k-j, k, j)$ for some positive integer $j$;

\item[$\bullet$] $k = i+j$ for some positive integers $i$ and $j$,
and $\la = (-i, -j, k)$.
\end{enumerate}
Similarly, if $k < 0$, then we have the following 2 cases:
\begin{enumerate}
\item[$\bullet$] $\la = (k, -j, -k + j)$ for some positive integer $j$;

\item[$\bullet$] $-k = i+j$ for some positive integers $i$ and $j$,
and $\la = (k, i, j)$.
\end{enumerate}
It follows from \eqref{Trala.05}, \eqref{Trala.06} and 
\eqref{Trala.07} that line \eqref{EquivCh1Ch1.simplify2} is equal to
$$
m^2 \cdot \left (-\sum_{k,j > 0} \frac{j (k+j)(q^{k+j}+q^{2k+j})}
     {(1-q^k)^2(1-q^j)(1-q^{k+j})}  - 
     \frac12 \sum_{i, j > 0} \frac{ij(q^{i+j}+q^{2i+2j})}
     {(1 - q^{i})(1 - q^{j})(1-q^{i+j})^2} \right )
$$
which is precisely $h_{1,1}^{(2)}(q) \cdot m^2$. Combining with 
\eqref{EquivCh1Ch1.simplify1}, \eqref{EquivCh1Ch1.simplify2} and
\eqref{EquivCh1Ch1.simplify3} completes the proof of 
\eqref{EquivCh1Ch1Prop.0}.
\end{proof}

By \eqref{Z246}, ${\bf QM} = \Q[Z(2), Z(4), Z(6)]$. 
By Lemma~\ref{EquivCh1Ch1Prop}, $h_{1,1}^{(4)}(q), h_{1,1}^{(2)}(q)$ 
and $h_{1,1}^{(0)}$ are quasi-modular forms of weight at most $6$. 
The next proposition is the main result of this section and 
expresses $h_{1,1}^{(4)}(q), h_{1,1}^{(2)}(q)$ 
and $h_{1,1}^{(0)}$ in terms of $Z(2), Z(4)$ and $Z(6)$.

\begin{proposition}    \label{h11024}
Let $h_{1,1}^{(0)}(q), h_{1,1}^{(2)}(q), h_{1,1}^{(4)}(q)$
be from Definition~\ref{h11}. Then, 
\begin{eqnarray}      
h_{1,1}^{(0)}(q) &=& Z(2)^2 + Z(4) - \frac{8}{3} Z(2)^3 + 4 Z(2)Z(4) 
                     + \frac{14}{3} Z(6),  \label{h11024.01}  \\
h_{1,1}^{(2)}(q) &=& \frac{5}{4} Z(2)^2 + \frac{5}{4} Z(4) 
                     - \frac{10}{3} Z(2)^3 + 5 Z(2)Z(4) 
                     + \frac{35}{6} Z(6),       \label{h11024.02}  \\
h_{1,1}^{(4)}(q) &=& -\frac{1}{4} Z(2)^2 - \frac{1}{4} Z(4) 
                     + \frac{2}{3} Z(2)^3 - Z(2)Z(4) 
                     - \frac{7}{6} Z(6).   \label{h11024.03}
\end{eqnarray}
\end{proposition}
\begin{proof}
We will only prove \eqref{h11024.01} since the proofs of \eqref{h11024.02} 
and \eqref{h11024.03} are similar. By Proposition~\ref{EquivCh1Ch1Prop}, 
$h_{1,1}^{(0)}$ is a quasi-modular forms of weight at most $6$.
By \cite[Proposition~2.7]{BK3} and \eqref{Z246}, the $\Q$-vector space of 
all quasi-modular forms of weight at most $6$ has 
a linear basis consisting of 
\begin{eqnarray}   \label{h11024.100}
1, Z(2), Z(2)^2, Z(4), Z(2)^3, Z(2)Z(4), Z(6).
\end{eqnarray}
So $h_{1,1}^{(0)}$ is a linear combination of \eqref{h11024.100} 
with coefficients in $\Q$. Note that 
\begin{eqnarray} 
%https://www.wolframalpha.com/widgets/view.jsp?id=f021a5566d8509939615e02a20f267e3     
Z(2) &=& q+3q^2+4q^3+7q^4+6q^5 + 12q^6 +8q^7 + O(q^8),  \nonumber  \\
%Z(2) &=& q/(1-q)^2+q^2/(1-q^2)^2+q^3/(1-q^3)^2+q^4/(1-q^4)^2+q^5/(1-q^5)^2+q^6/(1-q^6)^2+q^7/(1-q^7)^2,  \\
Z(2)^2 &=& q^2+6q^3+17q^4+38q^5+ 70q^6 +116q^7 +O(q^8),   
               \label{h11024.1}   \\
Z(4) &=& q^2+4q^3+11q^4+20q^5+ 40q^6 +56q^7 +O(q^8),  \label{h11024.2} \\
%Z(4) &=& q^2/(1-q)^4+q^4/(1-q^2)^4+q^6/(1-q^3)^4,  \\
Z(2)^3 &=& q^3+9q^4+39q^5+ 120q^6 +300q^7 +O(q^8),   \label{h11024.3}   \\
Z(2)Z(4) &=& q^3+7q^4+27q^5+ 76q^6 +178q^7 +O(q^8),  \label{h11024.4} \\
Z(6) &=& q^3+6q^4+21q^5+ 57q^6 +126q^7 +O(q^8).  \label{h11024.5}
%\\
%Z(6) &=& q^3/(1-q)^6+q^6/(1-q^2)^6.  \\    
\end{eqnarray}
Expanding the right-hand-side of \eqref{h11.01}, we obtain
\begin{eqnarray}    \label{h11024.6}
%https://www.wolframalpha.com/widgets/view.jsp?id=f021a5566d8509939615e02a20f267e3     
h_{1,1}^{(0)} &=& 2q^2+16q^3+60q^4+160q^5 + 360q^6 +672q^7 + O(q^8). 
%h_{1,1}^{(0)} &=& (2q^2)/((1-q)^2(1-q^2)) + (12q^3)/((1-q)(1-q^2)(1-q^3))\\
%              & & + (24q^4)/((1-q)(1-q^3)(1-q^4)) 
%                  + (16q^4)/((1-q^2)^2(1-q^4))   \\ 
%              & & + (40q^5)/((1-q)(1-q^4)(1-q^5))  
%                  + (60q^5)/((1-q^2)(1-q^3)(1-q^5))  \\
%              & & + (60q^6)/((1-q)(1-q^5)(1-q^6))  
%                  + (96q^6)/((1-q^2)(1-q^4)(1-q^6))  \\
%              & & + (54q^6)/((1-q^3)^2(1-q^6))    
%                  + (84q^7)/((1-q)(1-q^6)(1-q^7))  \\
%              & & + (140q^7)/((1-q^2)(1-q^5)(1-q^7))   
%                  + (168q^7)/((1-q^3)(1-q^4)(1-q^7)) \\
%h_{1,1}^{(1)} &=& q^2+6q^3+17q^4+38q^5+ 70q^6 +117q^7 +O(q^8),      \\
%h_{1,1}^{(2)} &=& q^3/(1-q)^6+q^6/(1-q^2)^6.    
\end{eqnarray}
Since $q|Z(2)$ but $q^2 \nmid Z(2)$, and $q^2$ divides $h_{1,1}^{(0)}, Z(4)$ and $Z(6)$, we have
\begin{eqnarray*}      
  h_{1,1}^{(0)} 
= r_1 \cdot Z(2)^2 + r_2 \cdot Z(4) + r_3 Z(2)^3 + r_4 Z(2)Z(4) + r_5 Z(6)
\end{eqnarray*}
for some $r_1, r_2, r_3, r_4, r_5 \in \Q$. Combining this with 
\eqref{h11024.1}-\eqref{h11024.6} enables us to 
determine the coefficients $r_1, r_2, r_3, r_4, r_5$:
$$
r_1=r_2=1, \quad r_3 = -\frac83, \quad r_4 = 4, \quad r_5 = \frac{14}3.
$$
This completes the proof of \eqref{h11024.01}.
\end{proof}

The following identities are not a prior obvious 
from Definition~\ref{h11}.

\begin{corollary}  \label{h11024-corollary}
Let $h_{1,1}^{(0)}(q), h_{1,1}^{(2)}(q), h_{1,1}^{(4)}(q)$
be from Definition~\ref{h11}. Then, 
$$
h_{1,1}^{(0)}(q) = \frac45 h_{1,1}^{(4)}(q) = -4 h_{1,1}^{(2)}(q).
$$
\end{corollary}
\begin{proof}
Follows immediately from Proposition~\ref{h11024}.
\end{proof}

\section{\bf Computation of $\langle \ch_1^{L_1} \ch_1^{L_2} \rangle'$} 
\label{sect_ch1L1ch1L2}

In this section, we will compute the reduced series 
$\langle \ch_1^{L_1} \ch_1^{L_2} \rangle'$ where $L_1$ and $L_2$ 
denote line bundles over a smooth projective surface $X$. 
We will prove that Qin's Conjecture~\ref{QinConj} holds for 
$\langle \ch_1^{L_1}\ch_1^{L_2} \rangle'$.

Recall from \eqref{OkoChkN.2} and \eqref{OkoChkN.1} that 
\begin{eqnarray*}
   \langle \ch_1^{L_1} \ch_1^{L_2} \rangle'
&=&(q;q)_\infty^{\chi(X)} \cdot 
   \langle \ch_1^{L_1} \ch_1^{L_2} \rangle     \\
&=&(q;q)_\infty^{\chi(X)} \cdot 
   \sum_{n \ge 0} q^n \int_{X^{[n]}} \ch_1\big ( L_1^{[n]} \big )
   \ch_1\big ( L_2^{[n]} \big ) c\big (T_{\Xn} \big ).
\end{eqnarray*}
By \cite[(5.41)]{Qin1}, 
$\ch_k (\Ln) = G_k(1_X, n) + G_{k-1}(L, n) + G_{k-2}(L^2/2, n)$. 
Combining with \eqref{F-generating}, we conclude that
\begin{eqnarray}   \label{ch1Lch1L.0}
   \langle \ch_1^{L_1} \ch_1^{L_2} \rangle'
= (q;q)_\infty^{\chi(X)} \cdot \big (F_{1,1}^{1_X, 1_X}(q)
  + F_{1,0}^{1_X, L_1}(q) + F_{1,0}^{1_X, L_2}(q)
  + F_{0,0}^{L_1, L_2}(q) \big ).
\end{eqnarray}
In the following, we will compute 
$(q;q)_\infty^{\chi(X)} \cdot F_{0,0}^{L_1, L_2}(q)$, 
$(q;q)_\infty^{\chi(X)} \cdot F_{1,0}^{1_X, L_i}(q)$ 
and $(q;q)_\infty^{\chi(X)} \cdot F_{1,1}^{1_X, 1_X}(q)$
in Lemma~\ref{F00L1L2}, Lemma~\ref{F101XL} and 
Lemma~\ref{F111X1X} respectively.

We begin with two lemmas which calculate various traces and 
are parallel to Lemma~\ref{Trala}. 

\begin{lemma}  \label{Tracei1Xj1X}
Let $i$ be a positive integer and $\alpha, \beta \in H^*(X)$. Then,
\begin{eqnarray}     
   \Tr \, q^\fn \fa_{-i}(\alpha)\fa_{i}(\beta) 
&=&(q;q)_\infty^{-\chi(X)} \cdot \langle \alpha, 
     \beta \rangle \frac{-i q^{i}}{1-q^i},  \label{Trij1Xij1X.01}  \\  
   \Tr \, q^\fn \fa_{i}(\alpha)\fa_{-i}(\beta) 
&=&(q;q)_\infty^{-\chi(X)} \cdot \langle \alpha, 
     \beta \rangle \frac{-i}{1-q^i},  \label{Trij1Xij1X.02}  \\ 
   \Tr \, q^\fn \fa_{-i}\fa_{i}(1_X) 
&=&(q;q)_\infty^{-\chi(X)} \cdot \chi(X) \frac{-iq^i}{1-q^i},  
     \label{Trij1Xij1X.03}  \\ 
   \Tr \, q^\fn \fa_{i}\fa_{-i}(1_X) 
&=&(q;q)_\infty^{-\chi(X)} \cdot \chi(X) \frac{-i}{1-q^i}.  
     \label{Trij1Xij1X.04}  
\end{eqnarray}
\end{lemma}
\begin{proof}
Formulas \eqref{Trij1Xij1X.01} and \eqref{Trij1Xij1X.02} follow from 
\eqref{TrqnfamuiNot0.2} and \eqref{TrqnfamuiNot0.3} respectively. 
To prove \eqref{Trij1Xij1X.03} and \eqref{Trij1Xij1X.04}, 
write $\tau_{2*}1_X = \sum_{s, t} \alpha_s \otimes \beta_t$. 
Then, 
$$
  \Tr \, q^\fn \fa_{-i}\fa_{i}(1_X)    
= \sum_{s, t} \Tr \, q^\fn \fa_{-i}(\alpha_s)\fa_{i}(\beta_t)
= \sum_{s, t} (q;q)_\infty^{-\chi(X)} \cdot \langle \alpha_s, 
     \beta_t \rangle \frac{-i q^{i}}{1-q^i},
$$
$$
  \Tr \, q^\fn \fa_{i}\fa_{-i}(1_X)    
= \sum_{s, t} \Tr \, q^\fn \fa_{i}(\alpha_s)\fa_{-i}(\beta_t)
= \sum_{s, t} (q;q)_\infty^{-\chi(X)} \cdot \langle \alpha_s, 
     \beta_t \rangle \frac{-i}{1-q^i}
$$
by \eqref{Trij1Xij1X.01} and \eqref{Trij1Xij1X.02}.
Now \eqref{Trij1Xij1X.03} and \eqref{Trij1Xij1X.04} follow from 
$\sum_{s, t} \alpha_s \cdot \beta_t = e_X$.
\end{proof}

\begin{lemma}  \label{Trij1Xij1X}
Let $i$ and $j$ be positive integers. Then,
\begin{eqnarray}     
   \Tr \, q^\fn \fa_{-i}\fa_{i+j}(1_X)\fa_{-i-j}\fa_{i}(1_X) 
&=&(q;q)_\infty^{-\chi(X)} \cdot \chi(X) 
     \frac{i(i+j) q^{i}}{(1-q^i)(1 - q^{i+j})},  \label{Trij1Xij1X.05}  \\
   \Tr \, q^\fn \fa_{-i-j}\fa_{i}(1_X)\fa_{-i}\fa_{i+j}(1_X) 
&=&(q;q)_\infty^{-\chi(X)} \cdot \chi(X) 
     \frac{i(i+j) q^{i+j}}{(1-q^i)(1 - q^{i+j})},  
     \label{Trij1Xij1X.06}  \\
   \Tr \, q^\fn \fa_i\fa_j(1_X)\fa_{-i}\fa_{-j}(1_X)
&=&(q;q)_\infty^{-\chi(X)} \cdot \chi(X) \frac{(1+\delta_{i,j})ij}{(1-q^{i})(1-q^{j})},
 %    \frac{(1+\delta_{i,j})ij}{(1 - q^{i})(1 - q^{j})}，
     \label{Trij1Xij1X.07}  \\
   \Tr \, q^\fn \fa_{-i}\fa_{-j}(1_X)\fa_i\fa_j(1_X)
&=&(q;q)_\infty^{-\chi(X)} \cdot \chi(X) 
     \frac{(1+\delta_{i,j})ijq^{i+j}}{(1 - q^{i})(1 - q^{j})}. 
     \label{Trij1Xij1X.08}
\end{eqnarray}
\end{lemma}
\begin{proof}
To prove \eqref{Trij1Xij1X.05}, we see from \eqref{Trqdjpx.3} that 
$$
  \Tr \, q^\fn \fa_{-i}\fa_{i+j}(1_X)\fa_{-i-j}\fa_{i}(1_X)
= \frac{1}{1 - q^{j}} \Tr \, q^\fn [\fa_{-i}\fa_{i+j}(1_X), 
   \fa_{-i-j}\fa_{i}(1_X)].
$$
By \eqref{QinLemma3.12(i)}, $[\fa_{-i}\fa_{i+j}(1_X), 
\fa_{-i-j}\fa_{i}(1_X)]$ is equal to
$$
i\fa_{-i-j}\fa_{i+j}(1_X) - (i+j)\fa_{-i}\fa_{i}(1_X).
$$
Thus, $\Tr \, q^\fn \fa_{-i}\fa_{i+j}(1_X)\fa_{-i-j}\fa_{i}(1_X)$ 
is equal to
\begin{eqnarray*}     
& &\frac{i}{1 - q^{j}} \Tr \, q^\fn \fa_{-i-j}\fa_{i+j}(1_X)
   - \frac{i+j}{1 - q^{j}} \Tr \, q^\fn \fa_{-i}\fa_{i}(1_X) \\
&=&\frac{i}{1 - q^{j}} (q;q)_\infty^{-\chi(X)} \cdot 
   \chi(X) \frac{-(i+j) q^{i+j}}{1 - q^{i+j}}
   - \frac{i+j}{1 - q^{j}} (q;q)_\infty^{-\chi(X)} \cdot 
   \chi(X) \frac{-i q^i}{1 - q^i}   \\
&=&(q;q)_\infty^{-\chi(X)} \cdot \chi(X) \cdot
   \frac{i(i+j) q^{i}}{(1-q^i)(1 - q^{i+j})}
\end{eqnarray*}
where we have used \eqref{Trij1Xij1X.01} in the first step. 
This verifies \eqref{Trij1Xij1X.05}. By \eqref{Trqdjpx.3},
\begin{eqnarray*}
   \Tr \, q^\fn \fa_{-i-j}\fa_{i}(1_X)\fa_{-i}\fa_{i+j}(1_X)
&=&\frac{q^{j}}{1 - q^{j}} \Tr \, q^\fn [\fa_{-i}\fa_{i+j}(1_X), 
     \fa_{-i-j}\fa_{i}(1_X)]   \\
&=&q^{j} \cdot \Tr \, q^\fn \fa_{-i}\fa_{i+j}(1_X)\fa_{-i-j}\fa_{i}(1_X).
\end{eqnarray*}
So \eqref{Trij1Xij1X.06} follows from \eqref{Trij1Xij1X.05}. 

Similarly, to prove \eqref{Trij1Xij1X.07}, we conclude from 
\eqref{Trqdjpx.3} that 
$$
  \Tr \, q^\fn \fa_i\fa_j(1_X)\fa_{-i}\fa_{-j}(1_X)
= \frac{1}{1 - q^{i+j}} \Tr \, q^\fn [\fa_i\fa_j(1_X), 
  \fa_{-i}\fa_{-j}(1_X)].
$$
By \eqref{QinLemma3.12(i)}, $[\fa_i\fa_j(1_X), \fa_{-i}\fa_{-j}(1_X)]$ 
is equal to
$$
(1+\delta_{i,j}) 
\big (-i\fa_{j}\fa_{-j}(1_X) - j\fa_{-i}\fa_{i}(1_X) \big ).
$$
Thus, $\Tr \, q^\fn \fa_i\fa_j(1_X)\fa_{-i}\fa_{-j}(1_X)$ 
is equal to
\begin{eqnarray*}     
& &-\frac{(1+\delta_{i,j})i}{1 - q^{i+j}} 
    \Tr \, q^\fn \fa_{j}\fa_{-j}(1_X) - 
    \frac{(1+\delta_{i,j})j}{1 - q^{i+j}} 
    \Tr \, q^\fn \fa_{-i}\fa_{i}(1_X) \\
&=&-\frac{(1+\delta_{i,j})i}{1 - q^{i+j}} (q;q)_\infty^{-\chi(X)} \cdot 
   \chi(X) \frac{-j}{1 - q^{j}} - 
    \frac{(1+\delta_{i,j})j}{1 - q^{i+j}} (q;q)_\infty^{-\chi(X)} \cdot 
   \chi(X) \frac{-iq^{i}}{1 - q^{i}}  \\
&=&(q;q)_\infty^{-\chi(X)} \cdot 
   \chi(X) \frac{(1+\delta_{i,j})ij}{(1 - q^{i})(1 - q^{j})}
\end{eqnarray*}
where we have used \eqref{Trij1Xij1X.04} and \eqref{Trij1Xij1X.03} 
in the first step. This verifies \eqref{Trij1Xij1X.07}. By \eqref{Trqdjpx.3},
\begin{eqnarray*}
   \Tr \, q^\fn \fa_{-i}\fa_{-j}(1_X)\fa_i\fa_j(1_X)
&=&\frac{q^{i+j}}{1 - q^{i+j}} \Tr \, q^\fn [\fa_i\fa_j(1_X), 
     \fa_{-i}\fa_{-j}(1_X)]   \\
&=&q^{i+j} \cdot \Tr \, q^\fn \fa_i\fa_j(1_X)\fa_{-i}\fa_{-j}(1_X).
\end{eqnarray*}
So \eqref{Trij1Xij1X.08} follows from \eqref{Trij1Xij1X.07}. 
\end{proof}

Next we present a formula for 
$(q;q)_\infty^{\chi(X)} \cdot F_{0,0}^{L_1, L_2}(q)$.

\begin{lemma}  \label{F00L1L2}
Let $L_1$ and $L_2$ be line bundles over a smooth projective surface $X$. 
Then, $(q;q)_\infty^{\chi(X)} \cdot F_{0,0}^{L_1, L_2}(q)$ is equal to
$$
Z(2)^2 \cdot \langle K_X, L_1 \rangle \langle K_X, L_2 \rangle 
   + \left (\frac72 Z(4) - \frac12 Z(2)^2 + Z(2) \right ) 
   \cdot \langle L_1, L_2 \rangle.
$$
\end{lemma}
\begin{proof}
By \eqref{FtoW.0} and \eqref{fG0L},
$$
  (q;q)_\infty^{\chi(X)} \cdot F_{0,0}^{L_1, L_2}(q) 
= (q;q)_\infty^{\chi(X)} \cdot \sum_{n, m > 0} 
  \Tr \, q^\fn \, {\bf W}(\fL_1, z) \, \fa_{-n}\fa_n(L_1)
  \fa_{-m}\fa_m(L_2).
$$
By \eqref{ThmJJkAlpha.3} and Lemma~\ref{TrqnfamuiNot0}, 
$(q;q)_\infty^{\chi(X)} \cdot F_{0,0}^{L_1, L_2}(q)$ is equal to 
\begin{eqnarray*}     
& &(q;q)_\infty^{\chi(X)} \cdot \left ( 
   (q;q)_\infty^{-\chi(X)} \cdot \sum_{n, m > 0} \frac{-q^n}{(1-q^n)^2}
   \frac{-q^m}{(1-q^m)^2} \langle -K_X, L_1 \rangle 
   \langle -K_X, L_2 \rangle \right.   \\
& &+ \sum_{n > 0} \frac{-q^n}{1-q^n} \frac{1}{1-q^n}\cdot 
   \Tr \, q^\fn \fa_{-n}((1_X - K_X)L_1)\fa_{n}(L_2)     \\
& &\left. + \sum_{n > 0} \frac{1}{1-q^n} \frac{-q^n}{1-q^n}\cdot 
   \Tr \, q^\fn \fa_{n}(L_1)\fa_{-n}((1_X - K_X)L_2)
   \right ).
\end{eqnarray*}
By \eqref{Trij1Xij1X.01} and \eqref{Trij1Xij1X.02},
\begin{eqnarray*}    
   \Tr \, q^\fn \fa_{-n}((1_X - K_X)L_1)\fa_{n}(L_2)
&=&(q;q)_\infty^{-\chi(X)} \cdot \frac{-nq^n}{1-q^n} 
      \langle L_1, L_2 \rangle,   \\
   \Tr \, q^\fn \fa_{n}(L_1)\fa_{-n}((1_X - K_X)L_2)
&=&(q;q)_\infty^{-\chi(X)} \cdot \frac{-n}{1-q^n} \langle L_1, L_2 \rangle.
\end{eqnarray*}
Therefore, $(q;q)_\infty^{\chi(X)} \cdot F_{0,0}^{L_1, L_2}(q)$ is equal to 
\begin{eqnarray*}     
& &\sum_{n, m > 0} \frac{-q^n}{(1-q^n)^2}
   \frac{-q^m}{(1-q^m)^2} \langle K_X, L_1 \rangle 
   \langle K_X, L_2 \rangle   \\
& &+ \sum_{n > 0} \frac{q^n}{(1-q^n)^2} \cdot \frac{nq^n}{1-q^n} 
  \langle L_1, L_2 \rangle  
  + \sum_{n > 0} \frac{q^n}{(1-q^n)^2} \cdot \frac{n}{1-q^n} 
  \langle L_1, L_2 \rangle   \\
&=&Z(2)^2 \cdot \langle K_X, L_1 \rangle 
   \langle K_X, L_2 \rangle + \sum_{n > 0} 
   \frac{nq^{2n} + nq^{n}}{(1-q^n)^3} \cdot \langle L_1, L_2 \rangle.
\end{eqnarray*}
Finally, our lemma follows from \cite[(5.11)]{AQ}.    
\end{proof}

Our next goal is to compute 
$(q;q)_\infty^{\chi(X)} \cdot F_{1,0}^{1_X, L}(q)$. 

\begin{lemma}  \label{F101XL}
Let $L$ be a line bundle over a smooth projective surface $X$. Then, 
$(q;q)_\infty^{\chi(X)} \cdot F_{1,0}^{1_X, L}(q)$ is equal to
$$   
\frac{1}{2} (Z(3) - Z(2)) \cdot Z(2) 
  \cdot K_X^2 \cdot \langle K_X, L \rangle   
+ \frac12 \cdot q\frac{\rm d}{{\rm d}q} (Z(3) - Z(2)) 
  \cdot \langle K_X, L \rangle.
$$
\end{lemma}
\begin{proof}
By \eqref{FtoW.0}, \eqref{char_thK=1.0} and \eqref{fG0L}, 
$(q;q)_\infty^{\chi(X)} \cdot F_{1,0}^{1_X, L}(q)$ 
is equal to
\begin{eqnarray}    \label{F101XL.1}
(q;q)_\infty^{\chi(X)} \cdot \sum_{\ell(\la) = 3, |\la| = 0, m > 0}  
  \Tr \, q^\fn \, {\bf W}(\fL_1, z) \, \frac{\fa_\la(1_X)}{\la^!} \, 
  \fa_{-m}\fa_m(L)
\end{eqnarray}
\begin{eqnarray}    \label{F101XL.2}
+ (q;q)_\infty^{\chi(X)} \cdot \frac{1}{2} \sum_{n, m > 0} (n-1)
  \Tr \, q^\fn \, {\bf W}(\fL_1, z) \, \fa_{-n} \fa_{n}(K_X) 
  \fa_{-m}\fa_m(L).
\end{eqnarray} 

By \eqref{ThmJJkAlpha.3} and Lemma~\ref{TrqnfamuiNot0}, 
line \eqref{F101XL.2} is equal to 
\begin{eqnarray*}     
& &(q;q)_\infty^{\chi(X)} \cdot \frac{1}{2} \left (
   (q;q)_\infty^{-\chi(X)} \cdot \sum_{n, m > 0}  
   \frac{-(n-1)q^n}{(1-q^n)^2} \frac{-q^m}{(1-q^m)^2} 
   \langle -K_X, K_X \rangle \langle -K_X, L \rangle \right.   \\
& &+ \sum_{n > 0} \frac{-(n-1)q^n}{1-q^n} \frac{1}{1-q^n}\cdot 
   \Tr \, q^\fn \fa_{-n}((1_X - K_X)K_X)\fa_{n}(L)     \\
& &\left. + \sum_{n > 0} \frac{n-1}{1-q^n} \frac{-q^n}{1-q^n}\cdot 
   \Tr \, q^\fn \fa_{n}(K_X)\fa_{-n}((1_X - K_X)L)
   \right ). 
\end{eqnarray*}
Combining with \eqref{Trij1Xij1X.01} and \eqref{Trij1Xij1X.02}, we see that
line \eqref{F101XL.2} is equal to 
\begin{eqnarray}     \label{F101XL.3}
& &\frac{1}{2} \left (\sum_{n > 0}  
   \frac{(n-1)q^n}{(1-q^n)^2} \cdot Z(2) \cdot 
   K_X^2 \langle K_X, L \rangle \right.   \nonumber  \\
& &+ \sum_{n > 0} \frac{(n-1)q^n}{1-q^n} \frac{1}{1-q^n}\cdot 
   \frac{nq^n}{1-q^n} \langle K_X, L \rangle   \nonumber  \\
& &\left. + \sum_{n > 0} \frac{n-1}{1-q^n} \frac{q^n}{1-q^n}\cdot 
   \frac{n}{1-q^n} \langle K_X, L \rangle \right )  \nonumber \\
&=&\frac{1}{2} \left (q\frac{\rm d}{{\rm d}q}[1] - Z(2) \right ) \cdot 
   Z(2) \cdot K_X^2 \langle K_X, L \rangle   \nonumber  \\
& &+ \frac{1}{2} \left (\sum_{n > 0} \frac{n(n-1)q^{2n}}{(1-q^n)^3} 
   + \sum_{n > 0} \frac{n(n-1)q^{n}}{(1-q^n)^3} \right ) 
   \langle K_X, L \rangle
\end{eqnarray}

Next, by \eqref{ThmJJkAlpha.3}, line \eqref{F101XL.1} is equal to 
$$
(q;q)_\infty^{\chi(X)} \cdot 
\sum_{\substack{\ell(\la) = 3\\|\la| = 0, m > 0}}
     \sum_{\substack{\w \la^{(1)} 
     \le \la^{(1)}, \w \la^{(2)} \le \la^{(2)}\\
     |\w \la^{(1)}| + |\w \la^{(2)}|= 0}}
     \,\, \prod_{\substack{1 \le i \le 2\\n \ge 1}} \left ( 
     \frac{(-1)^{p^{(i)}_{n}}}{p^{(i)}_{n}!} 
     \frac{q^{n p^{(i)}_{n}}}{(1-q^n)^{p^{(i)}_{n}}}  
     \frac{1}{\w p^{(i)}_{n}!} \frac{1}{(1-q^n)^{\w p^{(i)}_{n}}} \right )
$$
\begin{eqnarray}   \label{F101XL.4} 
\cdot \Tr \, q^\fn \prod_{i=1}^2  \frac{\fa_{\la^{(i)} - \w \la^{(i)}}
     \big ((1_X - K_X)^{\sum_{n \ge 1} p^{(i)}_{n}}\alpha_i \big )}
     {\big (\la^{(i)}-\w \la^{(i)} \big )^!}
\end{eqnarray}
where $\la^{(1)} = \la, \la^{(2)} = (-m, m), \alpha_1 = 1_X$ and 
$\alpha_2 = L$. By Lemma~\ref{TrqnfamuiNot0} and 
the convention \eqref{IntBeta}, the only nonzero contribution of 
the case $\w \la^{(2)} = \la^{(2)} = (-m, m)$ to \eqref{F101XL.4} is 
when $\w \la^{(1)} = \la^{(1)} = \la = (-i-j, i, j)$ for some positive 
integers $i$ and $j$. So the contribution of the case 
$\w \la^{(2)} = \la^{(2)} = (-m, m)$ to \eqref{F101XL.4} is equal to
\begin{eqnarray}   \label{F101XL.5}
& &\frac{1}{2} \sum_{i, j, m > 0} \frac{q^{i+j}}{(1-q^i)(1-q^j)(1-q^{i+j})}
   K_X^2 \, \frac{q^m}{(1-q^m)^2} \langle K_X, L \rangle   \nonumber \\
&=&\frac{1}{2} \left (Z(3) - q\frac{\rm d}{{\rm d}q}[1] \right ) \cdot 
   Z(2) \cdot K_X^2 \langle K_X, L \rangle
\end{eqnarray}
where we have used \cite[(5.13)]{AQ}. Similarly, 
the only nonzero contributions of the case $\w \la^{(2)} = (m) 
\le \la^{(2)} = (-m, m)$ to \eqref{F101XL.4} 
come from the following 2 situations:
\begin{enumerate}
\item[$\bullet$] $m = i+j$ for some positive integers $i$ and $j$, 
$\w \la^{(1)} = (-i, -j) \le \la^{(1)} = \la = (-i, -j, i + j)$;

\item[$\bullet$] $\w \la^{(1)} = (-i-m, i) \le \la^{(1)} = 
\la = (-i-m, i, m)$ 
for some positive integer $i$.
%
%\item[$\bullet$] $m = i+j$ for some positive integers $i$ and $j$, 
%$\w \la^{(1)} = (-i-j) \le \la = (-i-j, i,j)$;
%
%\item[$\bullet$] $\w \la^{(1)} = (-m) \le \la = (-i, -m, i+m)$ 
%for some positive integer $i$.
\end{enumerate}
So the contribution of the case $\w \la^{(2)} = (m) 
\le \la^{(2)} = (-m, m)$ to \eqref{F101XL.4} is equal to
\begin{eqnarray}   \label{F101XL.6}
& &(q;q)_\infty^{\chi(X)} \cdot 
     \left (\frac{1}{2} \sum_{i, j > 0} \frac{1}{1-q^i} \frac{1}{1-q^j} 
     \frac{-q^{i+j}}{1-q^{i+j}} \cdot 
     \Tr \, q^\fn \fa_{i+j}(1_X) \fa_{-i-j}((1_X - K_X)L) 
     \right.    \nonumber \\
& &+ \sum_{i \ne m, i, m > 0} \frac{-q^i}{1-q^i} \frac{1}{1-q^{i+m}} 
     \frac{-q^{m}}{1-q^{m}} \cdot 
     \Tr \, q^\fn \fa_{m}(1_X - K_X) \fa_{-m}((1_X - K_X)L) 
     \nonumber \\
& &+ \left. \sum_{i > 0} \frac{-q^i}{1-q^i} \frac{1}{1-q^{2i}} 
     \frac{-q^{i}}{1-q^{i}} \cdot 
     \Tr \, q^\fn \fa_{i}(1_X - K_X) \fa_{-i}((1_X - K_X)L) 
     \right )    \nonumber \\
&=&-\frac{1}{2} \sum_{i, j > 0} \frac{1}{1-q^i} \frac{1}{1-q^j} 
     \frac{q^{i+j}}{1-q^{i+j}} \cdot 
     \frac{i+j}{1-q^{i+j}} \langle K_X, L \rangle \nonumber \\
& &+ \sum_{i, m > 0} \frac{q^i}{1-q^i} \frac{1}{1-q^{i+m}} 
     \frac{q^{m}}{1-q^{m}} \cdot \frac{m}{1-q^m} \cdot
     2 \langle K_X, L \rangle
\end{eqnarray}
where we have used \eqref{Trij1Xij1X.02}. 
Likewise, the only nonzero contributions of the case $\w \la^{(2)} = (-m) 
\le \la^{(2)} = (-m, m)$ to \eqref{F101XL.4} 
come from the following 2 situations:
\begin{enumerate}
\item[$\bullet$] $m = i+j$ for some positive integers $i$ and $j$, 
$\w \la^{(1)} = (i, j) \le \la^{(1)} = \la = (-i -j, i, j)$;

\item[$\bullet$] $\w \la^{(1)} = (-i, i+m) \le \la^{(1)} = 
\la = (-i,-m, i+ m)$ 
for some positive integer $i$.
\end{enumerate}
Thus, the contribution of the case $\w \la^{(2)} = (-m) 
\le \la^{(2)} = (-m, m)$ to \eqref{F101XL.4} is 
\begin{eqnarray}   \label{F101XL.7}
& &\frac{1}{2} \sum_{i, j > 0} \frac{q^i}{1-q^i} \frac{q^j}{1-q^j} 
     \frac{1}{1-q^{i+j}} \cdot \frac{(i+j)q^{i+j}}{1-q^{i+j}} \cdot 
     2 \langle K_X, L \rangle \nonumber \\
& &- \sum_{i, m > 0} \frac{1}{1-q^i} \frac{q^{i+m}}{1-q^{i+m}} 
     \frac{1}{1-q^{m}} \cdot \frac{mq^m}{1-q^m} \cdot
     \langle K_X, L \rangle. 
\end{eqnarray}
Finally, the contribution of 
the case $\w \la^{(2)} = \emptyset$ to \eqref{F101XL.4} is $0$.
Combining with \eqref{F101XL.5}, \eqref{F101XL.6} and \eqref{F101XL.7}, 
we see that line \eqref{F101XL.1} is equal to
\begin{eqnarray}   \label{F101XL.8}
\frac{1}{2} \left (Z(3) - q\frac{\rm d}{{\rm d}q}[1] \right ) \cdot 
Z(2) \cdot K_X^2 \langle K_X, L \rangle   
+ A \cdot \langle K_X, L \rangle
\end{eqnarray}
where $A$ denotes 
\begin{eqnarray*}   
& &-\frac{1}{2} \sum_{i, j > 0} 
     \frac{(i+j)q^{i+j}}{(1-q^i)(1-q^j)(1-q^{i+j})^2}    
     + 2\sum_{i, j > 0} \frac{jq^{i+j}}{(1-q^i)(1-q^{j})^2(1-q^{i+j})} \\
& &+ \sum_{i, j > 0} 
     \frac{(i+j)q^{2(i+j)}}{(1-q^i)(1-q^j)(1-q^{i+j})^2}  
     - \sum_{i, j > 0} \frac{jq^{i+2j}}{(1-q^i)(1-q^j)^2(1-q^{i+j})}  \\
&=&-\frac{1}{2} \sum_{i, j > 0} 
     \frac{(i+j)q^{i+j}}{(1-q^i)(1-q^j)(1-q^{i+j})}    
     + \sum_{i, j > 0} \frac{jq^{i+j}}{(1-q^i)(1-q^{j})^2(1-q^{i+j})} \\
& &+ \frac{1}{2} \sum_{i, j > 0} 
     \frac{(i+j)q^{2(i+j)}}{(1-q^i)(1-q^j)(1-q^{i+j})^2}  
     + \sum_{i, j > 0} \frac{jq^{i+j}}{(1-q^i)(1-q^j)(1-q^{i+j})} \\
&=&\sum_{i, j > 0} \frac{jq^{i+j}}{(1-q^i)(1-q^{j})^2(1-q^{i+j})}
   + \frac{1}{2} \sum_{i, j > 0} 
     \frac{(i+j)q^{2(i+j)}}{(1-q^i)(1-q^j)(1-q^{i+j})^2}.
\end{eqnarray*}

Note that
\begin{eqnarray}   \label{qiqj}
  \frac{1}{(1-q^i)(1-q^j)}
= \left (\frac{1}{1-q^i} + \frac{q^j}{1-q^j} \right ) \frac{1}{1-q^{i+j}} 
\end{eqnarray}
Applying \eqref{qiqj} repeatedly, we conclude that $A$ is equal to
\begin{eqnarray*}   
& &\sum_{i, j > 0} \frac{jq^{i+2j}}{(1-q^{j})^2(1-q^{i+j})^2} 
   + \sum_{i, j > 0} \frac{jq^{i+j}}{(1-q^i)(1-q^{j})(1-q^{i+j})^2}  \\
& &+ \frac{1}{2} \sum_{i, j > 0} 
     \frac{(i+j)q^{2(i+j)}}{(1-q^i)(1-q^j)(1-q^{i+j})^2}   \\
&=&\sum_{i, j > 0} \frac{jq^{i+2j}}{(1-q^{j})^2(1-q^{i+j})^2} \\
& &+ \sum_{i, j > 0} \frac{jq^{i+j}}{(1-q^i)(1-q^{i+j})^3}  
   + \sum_{i, j > 0} \frac{jq^{i+2j}}{(1-q^j)(1-q^{i+j})^3}\\
& &+ \frac{1}{2} \sum_{i, j > 0} 
     \frac{(i+j)q^{2(i+j)}}{(1-q^i)(1-q^{i+j})^3}
     + \frac{1}{2} \sum_{i, j > 0} 
     \frac{(i+j)q^{2i+3j}}{(1-q^j)(1-q^{i+j})^3}.
\end{eqnarray*}
After some changes of variables, $A$ is equal to
\begin{eqnarray*}   
& &\sum_{n > m > 0} \frac{q^{n}}{(1-q^n)^2} \frac{mq^{m}}{(1-q^m)^2}\\
& &+ \sum_{n > m > 0} \frac{(n-m)q^n}{(1-q^n)^3} \frac{1}{1-q^m}  
   + \sum_{n > m > 0} \frac{q^{n}}{(1-q^n)^3} \frac{mq^{m}}{1-q^m}\\
& &+ \frac{1}{2} \sum_{n > m > 0} 
     \frac{nq^{2n}}{(1-q^n)^3} \frac{1}{1-q^m}
     + \frac{1}{2} \sum_{n > m > 0} 
     \frac{nq^{2n}}{(1-q^n)^3} \frac{q^{m}}{1-q^m}\\
&=&\sum_{n > m > 0} \frac{q^{n}}{(1-q^n)^2} \frac{mq^{m}}{(1-q^m)^2}
   + \sum_{n > m > 0} \frac{(n-m)q^n}{(1-q^n)^3} \frac{1}{1-q^m}\\
& &+ \sum_{n > m > 0} \frac{q^n}{(1-q^n)^3} \frac{m}{1-q^m}  
   - \frac{1}{2} \sum_{n > 0} \frac{n(n-1)q^{n}}{(1-q^n)^3}  \\
& &+ \sum_{n > m > 0} \frac{nq^{2n}}{(1-q^n)^3} \frac{1}{1-q^m}
   - \frac{1}{2} \sum_{n > 0} \frac{n(n-1)q^{2n}}{(1-q^n)^3}.
\end{eqnarray*}
Combining with \eqref{F101XL.1}, \eqref{F101XL.2}, \eqref{F101XL.3} 
and \eqref{F101XL.8}, we conclude that 
\begin{eqnarray}    \label{F101XL.9} 
  (q;q)_\infty^{\chi(X)} \cdot F_{1,0}^{1_X, L}(q)
= \frac{1}{2} (Z(3) - Z(2)) \cdot 
   Z(2) \cdot K_X^2 \cdot \langle K_X, L \rangle   
   + \W A \cdot \langle K_X, L \rangle 
\end{eqnarray}
where $\W A$ is equal to
\begin{eqnarray*}   
& &\sum_{n > m > 0} \frac{q^{n}}{(1-q^n)^2} \frac{mq^{m}}{(1-q^m)^2}
   + \sum_{n > m > 0} \frac{nq^{2n}+nq^n}{(1-q^n)^3} \frac{1}{1-q^m}\\
&=&q\frac{\rm d}{{\rm d}q} \sum_{n > m > 0} \frac{q^{n}}{(1-q^n)^2} \frac{1}{1-q^m}
\end{eqnarray*}
By \cite[Corollary~4]{Bra1}, we have
\begin{eqnarray}
\sum_{n_1 > n_2 > 0} \frac{q^{n_1}}{(1-q^{n_1})^2} \frac{1}{1-q^{n_2}}
= \sum_{n > 0} \frac{q^{2n}}{(1-q^n)^3}.   \label{Bra1Cor4.1}  
\end{eqnarray}
So $\W A$ is equal to
\begin{eqnarray*}   
& &q\frac{\rm d}{{\rm d}q} \sum_{n > 0} \frac{q^{2n}}{(1-q^n)^3}  \\
&=&q\frac{\rm d}{{\rm d}q} \frac12 \left (
   \sum_{n > 0} \frac{q^{2n} + q^{n}}{(1-q^n)^3}
   - \sum_{n > 0} \frac{q^{n} - q^{2n}}{(1-q^n)^3} \right ) \\
&=&\frac12 \cdot q\frac{\rm d}{{\rm d}q} (Z(3) - Z(2)).
\end{eqnarray*}
By \eqref{F101XL.9}, $(q;q)_\infty^{\chi(X)} \cdot F_{1,0}^{1_X, L}(q)$ 
is equal to
$$   
\frac{1}{2} (Z(3) - Z(2)) \cdot Z(2) 
  \cdot K_X^2 \cdot \langle K_X, L \rangle   
+ \frac12 \cdot q\frac{\rm d}{{\rm d}q} (Z(3) - Z(2)) 
  \cdot \langle K_X, L \rangle.
$$
This completes the proof of our lemma. 
\end{proof}

By \cite[Example~2.11]{BK3}, 
\begin{eqnarray}    \label{DZ3} 
q\frac{\rm d}{{\rm d}q} Z(3) = 5Z(5) - 4Z(3,2) - 6Z(2,3) + Z(3).
\end{eqnarray}

In the following, we calculate 
$(q;q)_\infty^{\chi(X)} \cdot F_{1,1}^{1_X, 1_X}(q)$.

\begin{lemma}  \label{F111X1X}
Let $X$ be a smooth projective surface. 
Let $h_{1,1}^{(2)}(q)$ and $h_{1,1}^{(4)}(q)$ be the quasi-modular forms 
from Definition~\ref{h11}. 
Then, $(q;q)_\infty^{\chi(X)} \cdot F_{1,1}^{1_X, 1_X}(q)$ is equal to
$$
\frac{1}{4} \big (Z(3) - Z(2) \big )^2 \cdot (K_X^2)^2 
+ h_{1,1}^{(2)}(q) \cdot \chi(X)
$$
$$
+ \left (\sum_{n > m > 0} \frac{q^n(1+q^{n})}{(1-q^n)^3} 
  \frac{n-nm+m^2}{1-q^m} 
  + 2\sum_{n > m > \ell > 0} \frac{nq^n(1+q^n)}{(1-q^n)^3}
     \frac{1}{1-q^m} \frac{1}{1-q^\ell}  \right.
$$
$$
\left. + 2 \sum_{n > m > \ell > 0} \frac{q^n}{(1-q^n)^2}
  \frac{mq^m}{(1-q^m)^2}\frac{1}{1-q^\ell} - h_{1,1}^{(4)}(q) \right )
  \cdot K_X^2.
$$
\end{lemma}
\begin{proof}
By \eqref{FtoW.0} and \eqref{char_thK=1.0},
$(q;q)_\infty^{\chi(X)} \cdot F_{1,1}^{1_X, 1_X}(q)$ is equal to
\begin{eqnarray}    \label{F111X1X.1}
(q;q)_\infty^{\chi(X)} \cdot \sum_{\substack{\ell(\la) = 3, |\la| = 0\\
  \ell(\mu) = 3, |\mu| = 0}}  \Tr \, q^\fn \, {\bf W}(\fL_1, z) \, 
  \frac{\fa_\la(1_X)}{\la^!} \, \frac{\fa_\mu(1_X)}{\mu^!}
\end{eqnarray}
\begin{eqnarray}    \label{F111X1X.2}
+ (q;q)_\infty^{\chi(X)} \cdot \frac12 
  \sum_{\ell(\la) = 3, |\la| = 0, m > 0} (m-1) \Tr \, q^\fn \, 
  {\bf W}(\fL_1, z) \, \frac{\fa_\la(1_X)}{\la^!} \fa_{-m}\fa_m(K_X)
\end{eqnarray} 
\begin{eqnarray}    \label{F111X1X.add}
+ (q;q)_\infty^{\chi(X)} \cdot \frac12 
  \sum_{\ell(\mu) = 3, |\mu| = 0, m > 0} (m-1) \Tr \, q^\fn \, 
  {\bf W}(\fL_1, z) \, \fa_{-m}\fa_m(K_X) \frac{\fa_\mu(1_X)}{\mu^!}
\end{eqnarray}
\begin{eqnarray}    \label{F111X1X.3}
+ (q;q)_\infty^{\chi(X)} \cdot \frac14 \sum_{n, m > 0} (n-1)(m-1)
  \Tr \, q^\fn \, {\bf W}(\fL_1, z) \, \fa_{-n} \fa_{n}(K_X) 
  \fa_{-m}\fa_m(K_X).
\end{eqnarray} 
 
First of all, by \eqref{ThmJJkAlpha.3} and Lemma~\ref{TrqnfamuiNot0}, 
line \eqref{F111X1X.3} is equal to
\begin{eqnarray*}     
& &(q;q)_\infty^{\chi(X)} \cdot \frac{1}{4} \left (
   (q;q)_\infty^{-\chi(X)} \cdot \sum_{n, m > 0}  
   \frac{-(n-1)q^n}{(1-q^n)^2} \frac{-(m-1)q^m}{(1-q^m)^2} 
   (-K_X^2) (-K_X^2) \right.   \\
& &+ \sum_{n > 0} \frac{-(n-1)^2q^n}{1-q^n} \frac{1}{1-q^n}\cdot 
   \Tr \, q^\fn \fa_{-n}((1_X - K_X)K_X)\fa_{n}(K_X)     \\
& &\left. + \sum_{n > 0} \frac{(n-1)^2}{1-q^n} \frac{-q^n}{1-q^n}\cdot 
   \Tr \, q^\fn \fa_{n}(K_X)\fa_{-n}((1_X - K_X)K_X)
   \right ). 
\end{eqnarray*}
Combining with \eqref{Trij1Xij1X.01} and \eqref{Trij1Xij1X.02}, we see that
line \eqref{F111X1X.3} is equal to 
\begin{eqnarray}     \label{F111X1X.4}
\frac{1}{4} \left (q\frac{\rm d}{{\rm d}q}[1] - Z(2) \right )^2 
    \cdot (K_X^2)^2   
+ \frac{1}{4} \sum_{n > 0} \frac{n(n-1)^2(q^{2n}+q^n)}{(1-q^n)^3} 
    \cdot K_X^2.
\end{eqnarray}

Next, by \eqref{ThmJJkAlpha.3}, line \eqref{F111X1X.add} is equal to 
$$
(q;q)_\infty^{\chi(X)} \cdot \frac12
\sum_{\substack{\ell(\mu) = 3, |\mu| = 0\\m > 0}} (m-1)
     \sum_{\substack{\w \la^{(1)} 
     \le \la^{(1)}, \w \la^{(2)} \le \la^{(2)}\\
     |\w \la^{(1)}| + |\w \la^{(2)}|= 0}}
$$
$$
\prod_{\substack{1 \le i \le 2\\n \ge 1}} \left ( 
     \frac{(-1)^{p^{(i)}_{n}}}{p^{(i)}_{n}!} 
     \frac{q^{n p^{(i)}_{n}}}{(1-q^n)^{p^{(i)}_{n}}}  
     \frac{1}{\w p^{(i)}_{n}!} \frac{1}{(1-q^n)^{\w p^{(i)}_{n}}} \right )
$$
\begin{eqnarray}   \label{F111X1X.add.1} 
\cdot \Tr \, q^\fn \prod_{i=1}^2  \frac{\fa_{\la^{(i)} - \w \la^{(i)}}
     \big ((1_X - K_X)^{\sum_{n \ge 1} p^{(i)}_{n}}\alpha_i \big )}
     {\big (\la^{(i)}-\w \la^{(i)} \big )^!}
\end{eqnarray}
where $\la^{(1)} = (-m, m), \la^{(2)} = \mu, \alpha_1 = K_X$ and 
$\alpha_2 = 1_X$. The contribution of the case 
$\w \la^{(1)} = \la^{(1)} = (-m, m)$ to \eqref{F111X1X.add} is equal to
\begin{eqnarray}   \label{F111X1X.add.2}
& &\frac{1}{2} \cdot 
   \frac{1}{2} \sum_{i, j, m > 0} \frac{(m-1)q^m}{(1-q^m)^2} K_X^2 \, 
   \frac{q^{i+j}}{(1-q^i)(1-q^j)(1-q^{i+j})} K_X^2    \nonumber \\
&=&\frac{1}{4} \left (Z(3) - q\frac{\rm d}{{\rm d}q}[1] \right ) \cdot 
   \left (q\frac{\rm d}{{\rm d}q}[1] - Z(2) \right ) \cdot (K_X^2)^2.
\end{eqnarray}
The contribution of the case $\w \la^{(1)} = (m) 
\le \la^{(1)} = (-m, m)$ to \eqref{F111X1X.add} is equal to
\begin{eqnarray}   \label{F111X1X.add.3}
%& &(q;q)_\infty^{\chi(X)} \cdot 
%     \left (\frac{1}{2} \sum_{i, j > 0} \frac{1}{1-q^i} \frac{1}{1-q^j} 
%     \frac{-(i+j-1)q^{i+j}}{1-q^{i+j}} \cdot 
%     \Tr \, q^\fn \fa_{i+j}(1_X) \fa_{-i-j}((1_X - K_X)K_X) 
%     \right.    \nonumber \\
%& &+ \sum_{i \ne m, i, m > 0} \frac{-q^i}{1-q^i} \frac{1}{1-q^{i+m}} 
%     \frac{-(m-1)q^{m}}{1-q^{m}} \cdot 
%     \Tr \, q^\fn \fa_{m}(1_X - K_X) \fa_{-m}((1_X - K_X)K_X) 
%     \nonumber \\
%& &+ \left. \sum_{i > 0} \frac{-q^i}{1-q^i} \frac{1}{1-q^{2i}} 
%     \frac{-(i-1)q^{i}}{1-q^{i}} \cdot 
%     \Tr \, q^\fn \fa_{i}(1_X - K_X) \fa_{-i}((1_X - K_X)K_X) 
%     \right )    \nonumber \\
& &-\frac{1}{4} \sum_{i, j > 0} \frac{1}{1-q^i} \frac{1}{1-q^j} 
     \frac{q^{i+j}}{1-q^{i+j}} \cdot 
     \frac{(i+j)(i+j-1)q^{i+j}}{1-q^{i+j}} \cdot K_X^2 \nonumber \\
& &+ \sum_{i, m > 0} \frac{q^i}{1-q^i} \frac{1}{1-q^{i+m}} 
     \frac{q^{m}}{1-q^{m}} \cdot \frac{m(m-1)q^{m}}{1-q^m} \cdot K_X^2.
\end{eqnarray}
The contribution of the case $\w \la^{(1)} = (-m) 
\le \la^{(1)} = (-m, m)$ to \eqref{F111X1X.add} is 
\begin{eqnarray}   \label{F111X1X.add.4}
& &\frac{1}{2} \sum_{i, j > 0} \frac{q^i}{1-q^i} \frac{q^j}{1-q^j} 
     \frac{1}{1-q^{i+j}} \cdot \frac{(i+j)(i+j-1)}{1-q^{i+j}} \cdot 
     K_X^2 \nonumber \\
& &- \frac{1}{2}\sum_{i, m > 0} \frac{1}{1-q^i} \frac{q^{i+m}}{1-q^{i+m}} 
     \frac{1}{1-q^{m}} \cdot \frac{m(m-1)}{1-q^m} \cdot K_X^2. 
\end{eqnarray}
Finally, the contribution of 
the case $\w \la^{(2)} = \emptyset$ to \eqref{F111X1X.add} is $0$.
Combining with \eqref{F111X1X.add.2}, \eqref{F111X1X.add.3} and 
\eqref{F111X1X.add.4}, we see that line \eqref{F111X1X.add} is equal to
\begin{eqnarray}   \label{F111X1X.add.5}
\frac{1}{4} \left (Z(3) - q\frac{\rm d}{{\rm d}q}[1] \right ) \cdot 
\left (q\frac{\rm d}{{\rm d}q}[1] - Z(2) \right ) \cdot (K_X^2)^2 
+ A \cdot K_X^2
\end{eqnarray}
where $A$ denotes 
\begin{eqnarray}   \label{F111X1X.add.6}  
& &-\frac{1}{4} \sum_{i, j > 0} 
     \frac{(i+j)(i+j-1)q^{2(i+j)}}{(1-q^i)(1-q^j)(1-q^{i+j})^2}    
     + \sum_{i, j > 0} 
     \frac{j(j-1)q^{i+2j}}{(1-q^i)(1-q^{j})^2(1-q^{i+j})} \nonumber \\
& &+ \frac{1}{2} \sum_{i, j > 0} 
     \frac{(i+j)(i+j-1)q^{i+j}}{(1-q^i)(1-q^j)(1-q^{i+j})^2}  
     - \frac{1}{2} \sum_{i, j > 0} 
     \frac{j(j-1)q^{i+j}}{(1-q^i)(1-q^j)^2(1-q^{i+j})}  \nonumber \\
&=&\frac{1}{4} \sum_{i, j > 0} 
     \frac{(i+j)(i+j-1)q^{i+j}}{(1-q^i)(1-q^j)(1-q^{i+j})}    
     + \frac{1}{2} \sum_{i, j > 0} 
     \frac{j(j-1)q^{i+2j}}{(1-q^i)(1-q^{j})^2(1-q^{i+j})} \nonumber \\
& &+ \frac{1}{4} \sum_{i, j > 0} 
     \frac{(i+j)(i+j-1)q^{i+j}}{(1-q^i)(1-q^j)(1-q^{i+j})^2}  
     - \frac{1}{2} \sum_{i, j > 0} 
     \frac{j(j-1)q^{i+j}}{(1-q^i)(1-q^j)(1-q^{i+j})}. \qquad 
\end{eqnarray}

Similarly, by \eqref{ThmJJkAlpha.3}, line \eqref{F111X1X.2} is equal to 
$$
(q;q)_\infty^{\chi(X)} \cdot \frac12
\sum_{\substack{\ell(\la) = 3, |\la| = 0\\m > 0}} (m-1)
     \sum_{\substack{\w \la^{(1)} 
     \le \la^{(1)}, \w \la^{(2)} \le \la^{(2)}\\
     |\w \la^{(1)}| + |\w \la^{(2)}|= 0}}
$$
$$
\prod_{\substack{1 \le i \le 2\\n \ge 1}} \left ( 
     \frac{(-1)^{p^{(i)}_{n}}}{p^{(i)}_{n}!} 
     \frac{q^{n p^{(i)}_{n}}}{(1-q^n)^{p^{(i)}_{n}}}  
     \frac{1}{\w p^{(i)}_{n}!} \frac{1}{(1-q^n)^{\w p^{(i)}_{n}}} \right )
$$
\begin{eqnarray}   \label{F111X1X.5} 
\cdot \Tr \, q^\fn \prod_{i=1}^2  \frac{\fa_{\la^{(i)} - \w \la^{(i)}}
     \big ((1_X - K_X)^{\sum_{n \ge 1} p^{(i)}_{n}}\alpha_i \big )}
     {\big (\la^{(i)}-\w \la^{(i)} \big )^!}
\end{eqnarray}
where $\la^{(1)} = \la, \la^{(2)} = (-m, m), \alpha_1 = 1_X$ and 
$\alpha_2 = K_X$. The contribution of the case 
$\w \la^{(2)} = \la^{(2)} = (-m, m)$ to \eqref{F111X1X.5} is equal to
\begin{eqnarray}   \label{F111X1X.6}
& &\frac{1}{2} \cdot 
   \frac{1}{2} \sum_{i, j, m > 0} \frac{q^{i+j}}{(1-q^i)(1-q^j)(1-q^{i+j})}
   K_X^2 \, \frac{(m-1)q^m}{(1-q^m)^2} K_X^2   \nonumber \\
&=&\frac{1}{4} \left (Z(3) - q\frac{\rm d}{{\rm d}q}[1] \right ) \cdot 
   \left (q\frac{\rm d}{{\rm d}q}[1] - Z(2) \right ) \cdot (K_X^2)^2.
\end{eqnarray}
The contribution of the case $\w \la^{(2)} = (m) 
\le \la^{(2)} = (-m, m)$ to \eqref{F111X1X.5} is equal to
\begin{eqnarray}   \label{F111X1X.7}
%& &(q;q)_\infty^{\chi(X)} \cdot 
%     \left (\frac{1}{2} \sum_{i, j > 0} \frac{1}{1-q^i} \frac{1}{1-q^j} 
%     \frac{-(i+j-1)q^{i+j}}{1-q^{i+j}} \cdot 
%     \Tr \, q^\fn \fa_{i+j}(1_X) \fa_{-i-j}((1_X - K_X)K_X) 
%     \right.    \nonumber \\
%& &+ \sum_{i \ne m, i, m > 0} \frac{-q^i}{1-q^i} \frac{1}{1-q^{i+m}} 
%     \frac{-(m-1)q^{m}}{1-q^{m}} \cdot 
%     \Tr \, q^\fn \fa_{m}(1_X - K_X) \fa_{-m}((1_X - K_X)K_X) 
%     \nonumber \\
%& &+ \left. \sum_{i > 0} \frac{-q^i}{1-q^i} \frac{1}{1-q^{2i}} 
%     \frac{-(i-1)q^{i}}{1-q^{i}} \cdot 
%     \Tr \, q^\fn \fa_{i}(1_X - K_X) \fa_{-i}((1_X - K_X)K_X) 
%     \right )    \nonumber \\
& &-\frac{1}{4} \sum_{i, j > 0} \frac{1}{1-q^i} \frac{1}{1-q^j} 
     \frac{q^{i+j}}{1-q^{i+j}} \cdot 
     \frac{(i+j)(i+j-1)}{1-q^{i+j}} \cdot K_X^2 \nonumber \\
& &+ \sum_{i, m > 0} \frac{q^i}{1-q^i} \frac{1}{1-q^{i+m}} 
     \frac{q^{m}}{1-q^{m}} \cdot \frac{m(m-1)}{1-q^m} \cdot K_X^2.
\end{eqnarray}
The contribution of the case $\w \la^{(2)} = (-m) 
\le \la^{(2)} = (-m, m)$ to \eqref{F111X1X.5} is 
\begin{eqnarray}   \label{F111X1X.8}
& &\frac{1}{2} \sum_{i, j > 0} \frac{q^i}{1-q^i} \frac{q^j}{1-q^j} 
     \frac{1}{1-q^{i+j}} \cdot \frac{(i+j)(i+j-1)q^{i+j}}{1-q^{i+j}} \cdot 
     K_X^2 \nonumber \\
& &- \frac{1}{2}\sum_{i, m > 0} \frac{1}{1-q^i} \frac{q^{i+m}}{1-q^{i+m}} 
     \frac{1}{1-q^{m}} \cdot \frac{m(m-1)q^m}{1-q^m} \cdot K_X^2. 
\end{eqnarray}
Finally, the contribution of 
the case $\w \la^{(2)} = \emptyset$ to \eqref{F111X1X.5} is $0$.
Combining with \eqref{F111X1X.6}, \eqref{F111X1X.7} and \eqref{F111X1X.8}, 
we see that line \eqref{F111X1X.2} is equal to
\begin{eqnarray}   \label{F111X1X.9}
\frac{1}{4} \left (Z(3) - q\frac{\rm d}{{\rm d}q}[1] \right ) \cdot 
\left (q\frac{\rm d}{{\rm d}q}[1] - Z(2) \right ) \cdot (K_X^2)^2 
+ B \cdot K_X^2
\end{eqnarray}
where $B$ denotes 
\begin{eqnarray}   \label{F111X1X.10}  
& &-\frac{1}{4} \sum_{i, j > 0} 
     \frac{(i+j)(i+j-1)q^{i+j}}{(1-q^i)(1-q^j)(1-q^{i+j})^2}    
     + \sum_{i, j > 0} 
     \frac{j(j-1)q^{i+j}}{(1-q^i)(1-q^{j})^2(1-q^{i+j})} \nonumber \\
& &+ \frac{1}{2} \sum_{i, j > 0} 
     \frac{(i+j)(i+j-1)q^{2(i+j)}}{(1-q^i)(1-q^j)(1-q^{i+j})^2}  
     - \frac{1}{2} \sum_{i, j > 0} 
     \frac{j(j-1)q^{i+2j}}{(1-q^i)(1-q^j)^2(1-q^{i+j})}  \nonumber \\
&=&-\frac{1}{4} \sum_{i, j > 0} 
     \frac{(i+j)(i+j-1)q^{i+j}}{(1-q^i)(1-q^j)(1-q^{i+j})}    
     + \frac{1}{2} \sum_{i, j > 0} 
     \frac{j(j-1)q^{i+j}}{(1-q^i)(1-q^{j})^2(1-q^{i+j})} \nonumber \\
& &+ \frac{1}{4} \sum_{i, j > 0} 
     \frac{(i+j)(i+j-1)q^{2(i+j)}}{(1-q^i)(1-q^j)(1-q^{i+j})^2}  
     + \frac{1}{2} \sum_{i, j > 0} 
     \frac{j(j-1)q^{i+j}}{(1-q^i)(1-q^j)(1-q^{i+j})}. \qquad
\end{eqnarray}

Line \eqref{F111X1X.1} is equal to 
$$
(q;q)_\infty^{\chi(X)} \cdot 
\sum_{\substack{\ell(\la) = 3, |\la| = 0\\\ell(\mu) = 3, |\mu| = 0}}
     \sum_{\substack{\w \la^{(1)} 
     \le \la^{(1)}, \w \la^{(2)} \le \la^{(2)}\\
     |\w \la^{(1)}| + |\w \la^{(2)}|= 0}}
     \,\, \prod_{\substack{1 \le i \le 2\\n \ge 1}} \left ( 
     \frac{(-1)^{p^{(i)}_{n}}}{p^{(i)}_{n}!} 
     \frac{q^{n p^{(i)}_{n}}}{(1-q^n)^{p^{(i)}_{n}}}  
     \frac{1}{\w p^{(i)}_{n}!} \frac{1}{(1-q^n)^{\w p^{(i)}_{n}}} \right )
$$
\begin{eqnarray}   \label{F111X1X.11} 
\cdot \Tr \, q^\fn \prod_{i=1}^2  \frac{\fa_{\la^{(i)} - \w \la^{(i)}}
     \big ((1_X - K_X)^{\sum_{n \ge 1} p^{(i)}_{n}} \big )}
     {\big (\la^{(i)}-\w \la^{(i)} \big )^!}
\end{eqnarray}
where $\la^{(1)} = \la$ and $\la^{(2)} = \mu$.
By Lemma~\ref{TrqnfamuiNot0} and 
the convention \eqref{IntBeta}, the only nonzero contribution of 
the case $\w \la^{(2)} = \la^{(2)} = \mu$ to \eqref{F111X1X.11} is 
when $\w \la^{(1)} = \la^{(1)} = \la = (-i-j, i, j)$ for some positive 
integers $i$ and $j$, and $\mu = (-k-\ell, k, \ell)$ for some positive 
integers $k$ and $\ell$. So the contribution of the case 
$\w \la^{(2)} = \la^{(2)}$ to \eqref{F111X1X.11} is 
\begin{eqnarray}   \label{F111X1X.12}
& &\frac{1}{4} \sum_{i, j, k, \ell > 0} 
   \frac{q^{i+j}}{(1-q^i)(1-q^j)(1-q^{i+j})} K_X^2 \, 
   \frac{q^{k+\ell}}{(1-q^k)(1-q^\ell)(1-q^{k+\ell})} K_X^2   \nonumber \\
&=&\frac{1}{4} \left (Z(3) - q\frac{\rm d}{{\rm d}q}[1] \right )^2  
   \cdot (K_X^2)^2.
\end{eqnarray}
The only nonzero contributions of the case $\w \la^{(2)} = (k, \ell) 
\le \la^{(2)}$ (with $k, \ell > 0$) to \eqref{F111X1X.11} 
come from the following 2 situations:
\begin{enumerate}
\item[$\bullet$] $k + \ell = i+j$ for some positive integers $i$ and $j$, 
$\w \la^{(1)} = (-i, -j) \le \la^{(1)} = \la = (-i, -j, i + j)$,
and $\la^{(2)} = \mu = (-k - \ell, k, \ell)$;

\item[$\bullet$] $\w \la^{(1)} = (-i-k-\ell, i) \le \la^{(1)} = 
\la = (-i-k-\ell, i, k + \ell)$ for some positive integer $i$, 
and $\la^{(2)} = \mu = (-k - \ell, k, \ell)$.
\end{enumerate}
Thus, the contribution of the case $\w \la^{(2)} = (k, \ell) 
\le \la^{(2)}$ (with $k, \ell > 0$) to \eqref{F111X1X.11} is 
\begin{eqnarray*}   
& &(q;q)_\infty^{\chi(X)} \cdot 
     \left (\frac{1}{4} \sum_{i+ j = k + \ell} 
     \frac{1}{1-q^i} \frac{1}{1-q^j} 
     \frac{-q^{k}}{1-q^k} \frac{-q^{\ell}}{1-q^\ell} \right.   \\
& &\qquad \cdot \Tr \, q^\fn \fa_{k + \ell}(1_X) 
     \fa_{-k - \ell}((1_X - K_X)^2) 
\end{eqnarray*}
\begin{eqnarray*}   
& &+ \frac{1}{2} \sum_{i, k, \ell > 0} \frac{-q^i}{1-q^i} 
     \frac{1}{1-q^{i+k + \ell}} 
     \frac{-q^{k}}{1-q^k} \frac{-q^{\ell}}{1-q^\ell} \\
& &\qquad \left. \cdot \Tr \, q^\fn \fa_{k + \ell}(1_X - K_X) 
     \fa_{-k - \ell}((1_X - K_X)^2) \right ).      
\end{eqnarray*}
By \eqref{Trij1Xij1X.02}, the contribution of 
$\w \la^{(2)} = (k, \ell) \le \la^{(2)}$ (with $k, \ell > 0$) 
to \eqref{F111X1X.11} is
$$
-\frac{1}{4} \sum_{i+ j = k + \ell} 
     \frac{1}{1-q^i} \frac{1}{1-q^j} 
     \frac{q^{k}}{1-q^k} \frac{q^{\ell}}{1-q^\ell} 
     \frac{k + \ell}{1-q^{k + \ell}} \cdot K_X^2
$$
\begin{eqnarray}   \label{F111X1X.13}
+ \frac{3}{2} \sum_{i, k, \ell > 0} \frac{q^i}{1-q^i} 
     \frac{1}{1-q^{i+k + \ell}} 
     \frac{q^{k}}{1-q^k} \frac{q^{\ell}}{1-q^\ell} 
     \frac{k + \ell}{1-q^{k + \ell}} \cdot K_X^2.   
\end{eqnarray}
The only nonzero contributions of the case $\w \la^{(2)} = (-k, \ell) 
\le \la^{(2)}$ (with $k, \ell > 0$) to \eqref{F111X1X.11} 
come from the following 3 situations:
\begin{enumerate}
\item[$\bullet$] $k = \ell + h$ for some positive integers $h$, 
$\la^{(2)} = \mu = (-k, \ell, h)$, and
$\w \la^{(1)} = (-i, i+h) \le \la^{(1)} = \la = (-h, -i, i + h)$
for some positive integer $i$;

\item[$\bullet$] $k = \ell + i+j$ for some positive integers $i$ and $j$, 
$\la^{(2)} = \mu = (-k, \ell, i+j)$, and
$\w \la^{(1)} = (i, j) \le \la^{(1)} = \la = (-i-j, i, j)$
for some positive integer $i$;

\item[$\bullet$] $\ell = k + h$ for some positive integers $h$, 
$\la^{(2)} = \mu = (-h, -k, \ell)$, and
$\w \la^{(1)} = (-i-h, i) \le \la^{(1)} = \la = (-i-h, i, h)$
for some positive integer $i$.
%
%\item[$\bullet$] $\ell = k + i+j$ for some positive integers $i$ and $j$,  
%$\la^{(2)} = \mu = (-i-j, -k, \ell)$, and
%$\w \la^{(1)} = (-i, -j) \le \la^{(1)} = \la = (-i,-j, i+j)$.
\end{enumerate}
So the contribution of the case $\w \la^{(2)} = (-k, \ell) 
\le \la^{(2)}$ (with $k, \ell > 0$) to \eqref{F111X1X.11} is 
$$
(q;q)_\infty^{\chi(X)} \cdot \left (\sum_{i, h, \ell > 0} 
     \frac{1}{1-q^i} \frac{-q^{i+h}}{1-q^{i+h}} 
     \frac{1}{1-q^{\ell + h}} \frac{-q^{\ell}}{1-q^\ell}  
\Tr \, q^\fn \fa_{-h}(1_X - K_X)\fa_{h}(1_X - K_X) \right.  
$$
$$
+ \frac12 \sum_{i, j, \ell > 0} 
     \frac{-q^i}{1-q^i} \frac{-q^j}{1-q^j} 
     \frac{1}{1-q^{i+j+\ell}} \frac{-q^{\ell}}{1-q^\ell}  
\Tr \, q^\fn \fa_{-i-j}((1_X - K_X)^2)\fa_{i+j}(1_X - K_X)
$$
$$
\left. + \sum_{i, k, h > 0} 
     \frac{1}{1-q^{i+h}} \frac{-q^i}{1-q^i} 
     \frac{1}{1-q^k} \frac{-q^{k+h}}{1-q^{k+h}}  
\Tr \, q^\fn \fa_{h}(1_X - K_X)\fa_{-h}(1_X - K_X) \right ).
$$
By \eqref{Trij1Xij1X.01} and \eqref{Trij1Xij1X.02},
the contribution of the case $\w \la^{(2)} = (-k, \ell) 
\le \la^{(2)}$ (with $k, \ell > 0$) to \eqref{F111X1X.11} is 
$$
-\sum_{i, h, \ell > 0} \frac{1}{1-q^i} \frac{q^{i+h}}{1-q^{i+h}} 
\frac{1}{1-q^{\ell + h}} \frac{q^{\ell}}{1-q^\ell}\frac{hq^h}{1-q^h} 
\cdot K_X^2 
$$
$$
+ \frac32 \sum_{i, j, \ell > 0} \frac{q^i}{1-q^i} \frac{q^j}{1-q^j} 
     \frac{1}{1-q^{i+j+\ell}} \frac{q^{\ell}}{1-q^\ell}  
\frac{(i+j)q^{i+j}}{1-q^{i+j}} \cdot K_X^2
$$
\begin{eqnarray}   \label{F111X1X.14}
- \sum_{i, k, h > 0} \frac{1}{1-q^{i+h}} \frac{q^i}{1-q^i} 
     \frac{1}{1-q^k} \frac{q^{k+h}}{1-q^{k+h}} \frac{h}{1-q^h} 
 \cdot K_X^2.
\end{eqnarray}
The only nonzero contribution of the case $\w \la^{(2)} = (-k, -\ell) 
\le \la^{(2)}$ (with $k, \ell > 0$) to \eqref{F111X1X.11} 
come from the following situation:
\begin{enumerate}
\item[$\bullet$] $k + \ell = i+j$ for some positive integers $i$ and $j$, 
$\la^{(2)} = \mu = (-k, -\ell, k + \ell)$,
and $\w \la^{(1)} = (i, j) \le \la^{(1)} = \la = (-i-j, i, j)$.
\end{enumerate}
So the contribution of the case $\w \la^{(2)} = (-k, -\ell) 
\le \la^{(2)}$ (with $k, \ell > 0$) to \eqref{F111X1X.11} is 
$$
(q;q)_\infty^{\chi(X)} \cdot \frac14 \sum_{i+j=k+\ell} 
     \frac{-q^i}{1-q^i} \frac{-q^j}{1-q^j} 
     \frac{1}{1-q^k} \frac{1}{1-q^\ell}  
\Tr \, q^\fn \fa_{-k - \ell}((1_X - K_X)^2)\fa_{k + \ell}(1_X).
$$
By \eqref{Trij1Xij1X.01},
the contribution of the case $\w \la^{(2)} = (-k, -\ell) 
\le \la^{(2)}$ (with $k, \ell > 0$) to \eqref{F111X1X.11} is 
\begin{eqnarray}   \label{F111X1X.15}
-\frac14 \sum_{i+j=k+\ell} \frac{q^i}{1-q^i} \frac{q^j}{1-q^j} 
\frac{1}{1-q^k} \frac{1}{1-q^\ell} \frac{(k+\ell)q^{k+\ell}}{1-q^{k+\ell}} 
\cdot K_X^2.
\end{eqnarray}
The only nonzero contributions of the case $\w \la^{(2)} = (k) 
\le \la^{(2)}$ (with $k > 0$) to \eqref{F111X1X.11} 
come from the following 2 situations:
\begin{enumerate}
\item[$\bullet$] $\la^{(2)} = \mu = (-k-\ell, \ell, k)$ 
for some positive integer $\ell$,
and $\w \la^{(1)} = (-k) \le \la^{(1)} = \la = (-k, -\ell, k+\ell)$;

\item[$\bullet$] $k = i+j$ for some positive integers $i$ and $j$, 
$\la^{(2)} = \mu = (-i, -j, k)$,
and $\w \la^{(1)} = (-k) \le \la^{(1)} = \la = (-k, i, j)$.
\end{enumerate}
So the contribution of the case $\w \la^{(2)} = (k) 
\le \la^{(2)}$ (with $k > 0$) to \eqref{F111X1X.11} is 
\begin{eqnarray*}   
& &(q;q)_\infty^{\chi(X)} \cdot \left (\sum_{k, \ell > 0} \frac{1}{1-q^k} 
   \frac{-q^k}{1-q^k} \Tr \, q^\fn \fa_{-\ell}\fa_{k+\ell}(1_X)
   \fa_{-k - \ell}\fa_\ell(1_X - K_X) \right.    \\
& &+ \sum_{i> j > 0} \frac{1}{1-q^{i+j}} 
   \frac{-q^{i+j}}{1-q^{i+j}} \Tr \, q^\fn \fa_i\fa_j(1_X)
   \fa_{-i}\fa_{-j}(1_X - K_X)   \\
& &\left. + \frac14 \sum_{i > 0} \frac{1}{1-q^{2i}} 
   \frac{-q^{2i}}{1-q^{2i}} \Tr \, q^\fn \fa_i\fa_i(1_X)
   \fa_{-i}\fa_{-i}(1_X - K_X) \right )   \\
&=&(q;q)_\infty^{\chi(X)} \cdot \left (-\sum_{k, \ell > 0}  
   \frac{q^k}{(1-q^k)^2} \Tr \, q^\fn \fa_{-\ell}\fa_{k+\ell}(1_X)
   \fa_{-k - \ell}\fa_\ell(1_X) \right.    \\
& &- \sum_{i> j > 0}  
   \frac{q^{i+j}}{(1-q^{i+j})^2} \Tr \, q^\fn \fa_i\fa_j(1_X)
   \fa_{-i}\fa_{-j}(1_X)   \\
& &\left. - \frac14 \sum_{i > 0}   
   \frac{q^{2i}}{(1-q^{2i})^2} \Tr \, q^\fn \fa_i\fa_i(1_X)
   \fa_{-i}\fa_{-i}(1_X) \right ).
\end{eqnarray*} 
By \eqref{Trij1Xij1X.05} and \eqref{Trij1Xij1X.07},
the contribution of the case $\w \la^{(2)} = (k) 
\le \la^{(2)}$ (with $k > 0$) to \eqref{F111X1X.11} is equal to 
\begin{eqnarray}   \label{F111X1X.16}
& &-\chi(X) \cdot \sum_{k, \ell > 0} \frac{q^k}{(1-q^k)^2} 
   \frac{\ell (k+\ell)q^\ell}{(1-q^\ell)(1-q^{k+\ell})}   \nonumber  \\
& &-\chi(X) \cdot \frac12 \sum_{i, j > 0} \frac{q^{i+j}}{(1-q^{i+j})^2} 
   \frac{ij}{(1 - q^{i})(1 - q^{j})}.
\end{eqnarray} 
The only nonzero contributions of the case $\w \la^{(2)} = (-k) 
\le \la^{(2)}$ (with $k > 0$) to \eqref{F111X1X.11} 
come from the following 2 situations:
\begin{enumerate}
\item[$\bullet$] $\la^{(2)} = \mu = (-k, -\ell, k+\ell)$ 
for some positive integer $\ell$,
and $\w \la^{(1)} = (k) \le \la^{(1)} = \la = (-k-\ell, \ell, k)$;

\item[$\bullet$] $k = i+j$ for some positive integers $i$ and $j$, 
$\la^{(2)} = \mu = (-k, i, j)$,
and $\w \la^{(1)} = (k) \le \la^{(1)} = \la = (-i, -j, k)$.
\end{enumerate}
So the contribution of the case $\w \la^{(2)} = (-k) 
\le \la^{(2)}$ (with $k > 0$) to \eqref{F111X1X.11} is 
\begin{eqnarray*}   
& &(q;q)_\infty^{\chi(X)} \cdot \left (\sum_{k, \ell > 0}  
   \frac{-q^k}{1-q^k} \frac{1}{1-q^k} 
   \Tr \, q^\fn \fa_{-k - \ell}\fa_\ell(1_X - K_X)
   \fa_{-\ell}\fa_{k+\ell}(1_X) \right.    \\
& &+ \sum_{i> j > 0} \frac{-q^{i+j}}{1-q^{i+j}} \frac{1}{1-q^{i+j}} 
   \Tr \, q^\fn \fa_{-i}\fa_{-j}(1_X - K_X)\fa_i\fa_j(1_X) \\
& &\left. + \frac14 \sum_{i > 0} \frac{-q^{2i}}{1-q^{2i}} 
   \frac{1}{1-q^{2i}} \Tr \, q^\fn 
   \fa_{-i}\fa_{-i}(1_X - K_X) \fa_i\fa_i(1_X) \right )   \\
&=&(q;q)_\infty^{\chi(X)} \cdot \left (-\sum_{k, \ell > 0}  
   \frac{q^k}{(1-q^k)^2} 
   \Tr \, q^\fn \fa_{-k - \ell}\fa_\ell(1_X)
   \fa_{-\ell}\fa_{k+\ell}(1_X) \right.    \\
& &- \sum_{i> j > 0} \frac{q^{i+j}}{(1-q^{i+j})^2} 
   \Tr \, q^\fn \fa_{-i}\fa_{-j}(1_X)\fa_i\fa_j(1_X) \\
& &\left. - \frac14 \sum_{i > 0} \frac{q^{2i}}{(1-q^{2i})^2} 
   \Tr \, q^\fn \fa_{-i}\fa_{-i}(1_X) \fa_i\fa_i(1_X) \right ).
\end{eqnarray*} 
By \eqref{Trij1Xij1X.06} and \eqref{Trij1Xij1X.08},
the contribution of the case $\w \la^{(2)} = (k) 
\le \la^{(2)}$ (with $k > 0$) to \eqref{F111X1X.11} is equal to 
\begin{eqnarray}   \label{F111X1X.17}
& &-\chi(X) \cdot \sum_{k, \ell > 0} \frac{q^k}{(1-q^k)^2} 
   \frac{\ell (k+\ell)q^{k+\ell}}{(1-q^\ell)(1-q^{k+\ell})} \nonumber  \\
& &-\chi(X) \cdot \frac12 \sum_{i, j > 0} \frac{q^{i+j}}{(1-q^{i+j})^2} 
   \frac{ijq^{i+j}}{(1 - q^{i})(1 - q^{j})}.
\end{eqnarray} 
Finally, the contribution of the case $\w \la^{(2)} = \emptyset 
\le \la^{(2)}$ to \eqref{F111X1X.11} is $0$. 
Combining with the formulas \eqref{F111X1X.11}-\eqref{F111X1X.17}, 
we see that line \eqref{F111X1X.1} is equal to
\begin{eqnarray}   \label{F111X1X.18}
\frac{1}{4} \left (Z(3) - q\frac{\rm d}{{\rm d}q}[1] \right )^2  
   \cdot (K_X^2)^2 + h_{1,1}^{(2)}(q) \cdot \chi(X) + C \cdot K_X^2
\end{eqnarray} 
where $h_{1,1}^{(2)}(q)$ is the quasi-modular form 
from Definition~\ref{h11}, and $C$ denotes
$$
-\frac{1}{4} \sum_{i,j,k, \ell >0, i+j=k+\ell} \frac{(i+j)q^{i+j}
(1+q^{i+j})}{(1-q^i)(1-q^j)(1-q^k)(1-q^{\ell})(1-q^{i+j})}
$$
$$
+ \frac{3}{2} \sum_{i, j, k > 0} \frac{(i+j)q^{i+j+k}(1+q^{i+j})}
{(1-q^i)(1-q^j)(1-q^k)(1-q^{i+j})(1-q^{i+j+k})}    
$$
$$
- \sum_{i,j,k >0} \frac{kq^{i+j+k}(1+q^{k})}
{(1-q^i)(1-q^j)(1-q^k)(1-q^{i+k})(1-q^{j+k})}.
$$

Combining \eqref{F111X1X.1}, \eqref{F111X1X.2}, \eqref{F111X1X.add},
\eqref{F111X1X.3}, \eqref{F111X1X.4}, \eqref{F111X1X.add.5}, 
\eqref{F111X1X.9} and \eqref{F111X1X.18}, 
we conclude that $(q;q)_\infty^{\chi(X)} \cdot F_{1,1}^{1_X, 1_X}(q)$ 
is equal to
\begin{eqnarray}   \label{F111X1X.19}
\frac{1}{4} \big (Z(3) - Z(2) \big )^2 \cdot (K_X^2)^2 
+ h_{1,1}^{(2)}(q) \cdot \chi(X) + D \cdot K_X^2
\end{eqnarray}
where $D$ is equal to
$$
\frac{1}{4} \sum_{n > 0} \frac{n(n-1)^2(q^{2n}+q^n)}{(1-q^n)^3}
$$
$$
+ \frac{1}{2} \sum_{i, j > 0} 
     \frac{j(j-1)q^{i+j}(1+q^{j})}{(1-q^i)(1-q^{j})^2(1-q^{i+j})}
+ \frac{1}{4} \sum_{i, j > 0} 
     \frac{(i+j)(i+j-1)q^{i+j}(1+q^{i+j})}{(1-q^i)(1-q^j)(1-q^{i+j})^2}
$$
$$
-\frac{1}{4} \sum_{i,j,k, \ell >0, i+j=k+\ell} \frac{(i+j)q^{i+j}
(1+q^{i+j})}{(1-q^i)(1-q^j)(1-q^k)(1-q^{\ell})(1-q^{i+j})}
$$
$$
+ \frac{3}{2} \sum_{i, j, k > 0} \frac{(i+j)q^{i+j+k}(1+q^{i+j})}
{(1-q^i)(1-q^j)(1-q^k)(1-q^{i+j})(1-q^{i+j+k})}    
$$
$$
- \sum_{i,j,k >0} \frac{kq^{i+j+k}(1+q^{k})}
{(1-q^i)(1-q^j)(1-q^k)(1-q^{i+k})(1-q^{j+k})}.
$$
By Definition~\ref{h11}, $D$ is equal to
$$
\frac{1}{4} \sum_{n > 0} \frac{n(n-1)^2(q^{2n}+q^n)}{(1-q^n)^3}
$$
$$
+ \frac{1}{2} \sum_{i, j > 0} 
     \frac{j(j-1)q^{i+j}(1+q^{j})}{(1-q^i)(1-q^{j})^2(1-q^{i+j})}
+ \frac{1}{4} \sum_{i, j > 0} 
     \frac{(i+j)(i+j-1)q^{i+j}(1+q^{i+j})}{(1-q^i)(1-q^j)(1-q^{i+j})^2}
$$
$$
+ \frac12 \sum_{i, j, k > 0} \frac{(i+j)q^{i+j+k}(1+q^{i+j})}
{(1-q^i)(1-q^j)(1-q^k)(1-q^{i+j})(1-q^{i+j+k})} - h_{1,1}^{(4)}(q).    
$$
Applying \eqref{qiqj} repeatedly, we conclude that $D$ is equal to
\begin{eqnarray*}   
& &\frac{1}{4} \sum_{n > 0} \frac{n(n-1)^2(q^{2n}+q^n)}{(1-q^n)^3}  \\
& &+ \frac{1}{2} \sum_{i, j > 0} 
     \frac{j(j-1)q^{i+j}(1+q^{j})}{(1-q^i)(1-q^{j})(1-q^{i+j})^2} 
     + \frac{1}{2} \sum_{i, j > 0} 
     \frac{j(j-1)q^{i+2j}(1+q^{j})}{(1-q^{j})^2(1-q^{i+j})^2} \\
& &+ \frac{1}{4} \sum_{i, j > 0} 
     \frac{(i+j)(i+j-1)q^{i+j}(1+q^{i+j})}{(1-q^i)(1-q^{i+j})^3}   \\
& &+ \frac{1}{4} \sum_{i, j > 0} 
     \frac{(i+j)(i+j-1)q^{i+2j}(1+q^{i+j})}{(1-q^j)(1-q^{i+j})^3} \\
& &+ \frac12 \sum_{i, j, k > 0} \frac{(i+j)q^{i+j+k}(1+q^{i+j})}
     {(1-q^i)(1-q^k)(1-q^{i+j})^2(1-q^{i+j+k})}  \\
& &+ \frac12 \sum_{i, j, k > 0} \frac{(i+j)q^{i+2j+k}(1+q^{i+j})}
     {(1-q^j)(1-q^k)(1-q^{i+j})^2(1-q^{i+j+k})} - h_{1,1}^{(4)}(q)  \\
&=&\frac{1}{4} \sum_{n > 0} \frac{n(n-1)^2(q^{2n}+q^n)}{(1-q^n)^3}  
   + \sum_{i, j > 0} 
     \frac{j(j-1)q^{i+j}}{(1-q^i)(1-q^{j})(1-q^{i+j})^2} \\
& &- \frac{1}{2} \sum_{i, j > 0} 
     \frac{j(j-1)q^{i+j}}{(1-q^i)(1-q^{i+j})^2} 
     + \frac{1}{2} \sum_{n > m > 0} \frac{q^n}{(1-q^n)^2}
     \frac{m(m-1)q^{m}(1+q^{m})}{(1-q^{m})^2} \\
& &+ \frac{1}{4} \sum_{n > m > 0} \frac{n(n-1)q^{n}(1+q^{n})}{(1-q^n)^3}
     \frac{1}{1-q^m} + \frac{1}{4} \sum_{n > m > 0} 
     \frac{n(n-1)q^{n}(1+q^{n})}{(1-q^n)^3} \frac{q^m}{1-q^m} \\
& &+ \frac12 \sum_{i, j, k > 0} \frac{(i+j)q^{i+j+k}(1+q^{i+j})}
     {(1-q^i)(1-q^k)(1-q^{i+j})(1-q^{i+j+k})^2}  \\
& &+ \frac12 \sum_{i, j, k > 0} \frac{(i+j)q^{2i+2j+k}(1+q^{i+j})}
     {(1-q^i)(1-q^{i+j})^2(1-q^{i+j+k})^2}    \\
& &+ \frac12 \sum_{i, j, k > 0} \frac{(i+j)q^{i+2j+k}(1+q^{i+j})}
     {(1-q^j)(1-q^k)(1-q^{i+j})(1-q^{i+j+k})^2}   \\
& &+ \frac12 \sum_{i, j, k > 0} \frac{(i+j)q^{2i+3j+k}(1+q^{i+j})}
     {(1-q^j)(1-q^{i+j})^2(1-q^{i+j+k})^2} - h_{1,1}^{(4)}(q).
\end{eqnarray*}
A long and tedious simplification shows that $D$ is equal to
$$
\sum_{n > m > 0} \frac{q^n(1+q^{n})}{(1-q^n)^3} \frac{n-nm+m^2}{1-q^m} 
     + 2\sum_{n > m > \ell > 0} \frac{nq^n(1+q^n)}{(1-q^n)^3}
     \frac{1}{1-q^m} \frac{1}{1-q^\ell}
$$
$$
+ 2 \sum_{n > m > \ell > 0} \frac{q^n}{(1-q^n)^2}
     \frac{mq^m}{(1-q^m)^2}\frac{1}{1-q^\ell} - h_{1,1}^{(4)}(q).
$$
Now our lemma follows from \eqref{F111X1X.19}.
\end{proof}
%Note that $\sum_{m=1}^{n-1} m^2 = n(n-1)(2n-1)/6$.

\begin{theorem} \label{theorem_ch1Lch1L}
Let $L_1$ and $L_2$ be line bundles over a smooth projective surface $X$. 
Then, the reduced generating series 
$\langle \ch_1^{L_1}\ch_1^{L_2} \rangle'$ is equal to
$$
Z(2)^2 \cdot \langle K_X, L_1 \rangle \langle K_X, L_2 \rangle 
   + \left (\frac72 Z(4) - \frac12 Z(2)^2 + Z(2) \right ) 
   \cdot \langle L_1, L_2 \rangle
$$
$$   
+ \frac{1}{2} (Z(3) - Z(2)) \cdot Z(2) 
  \cdot K_X^2 \cdot \langle K_X, L_1 \rangle   
+ \frac12 \cdot q\frac{\rm d}{{\rm d}q} (Z(3) - Z(2)) 
  \cdot \langle K_X, L_1 \rangle
$$
$$   
+ \frac{1}{2} (Z(3) - Z(2)) \cdot Z(2) 
  \cdot K_X^2 \cdot \langle K_X, L_2 \rangle   
+ \frac12 \cdot q\frac{\rm d}{{\rm d}q} (Z(3) - Z(2)) 
  \cdot \langle K_X, L_2 \rangle
$$
$$
+ \left (\frac{5}{4} Z(2)^2 + \frac{5}{4} Z(4) - \frac{10}{3} Z(2)^3 
+ 5 Z(2)Z(4) + \frac{35}{6} Z(6) \right ) \cdot \chi(X)
$$
$$
+ \frac{1}{4} \big (Z(3) - Z(2) \big )^2 \cdot (K_X^2)^2 
+ \left (\sum_{n > m > 0} \frac{q^n(1+q^{n})}{(1-q^n)^3} 
  \frac{n-nm+m^2}{1-q^m}   \right.
$$
$$
+ 2\sum_{n > m > \ell > 0} \frac{nq^n(1+q^n)}{(1-q^n)^3}
     \frac{1}{1-q^m} \frac{1}{1-q^\ell}
+ 2 \sum_{n > m > \ell > 0} \frac{q^n}{(1-q^n)^2}
  \frac{mq^m}{(1-q^m)^2}\frac{1}{1-q^\ell} 
$$
$$
\left. + \frac{1}{4} Z(2)^2 + \frac{1}{4} Z(4) - \frac{2}{3} Z(2)^3 
+ Z(2)Z(4) + \frac{7}{6} Z(6) \right ) \cdot K_X^2.
$$
\end{theorem}
\begin{proof}
Follows immediately from \eqref{ch1Lch1L.0}, Lemma~\ref{F00L1L2},
Lemma~\ref{F101XL}, Lemma~\ref{F111X1X} and Proposition~\ref{h11024}.
\end{proof}

Now we are ready to prove the main result in this paper. 

\begin{theorem} \label{corollary_ch1Lch1L}
Let $L_1$ and $L_2$ be line bundles over a smooth projective surface $X$.
If the canonical class of $X$ is numerically trivial, 
then $\langle \ch_1^{L_1}\ch_1^{L_2} \rangle'$ is equal to
$$
\left (\frac72 Z(4) - \frac12 Z(2)^2 + Z(2) \right ) \cdot 
\langle L_1, L_2 \rangle
$$
$$
+ \left (\frac{5}{4} Z(2)^2 + \frac{5}{4} Z(4) - \frac{10}{3} Z(2)^3 
+ 5 Z(2)Z(4) + \frac{35}{6} Z(6) \right ) \cdot \chi(X).
$$
In particular, Qin's Conjecture~\ref{QinConj} holds for 
$\langle \ch_1^{L_1}\ch_1^{L_2} \rangle'$.
\end{theorem}
\begin{proof}
By Theorem~\ref{theorem_ch1Lch1L}, 
$\langle \ch_1^{L_1}\ch_1^{L_2} \rangle'$ is equal to
$$
\left (\frac72 Z(4) - \frac12 Z(2)^2 + Z(2) \right ) \cdot 
\langle L_1, L_2 \rangle
$$
$$
+ \left (\frac{5}{4} Z(2)^2 + \frac{5}{4} Z(4) - \frac{10}{3} Z(2)^3 
+ 5 Z(2)Z(4) + \frac{35}{6} Z(6) \right ) \cdot \chi(X).
$$
Note that $\chi(X)$ and $\langle L_1, L_2 \rangle$ are integers. 
By \eqref{Z246}, 
$\langle \ch_1^{L_1}\ch_1^{L_2} \rangle'$ is a quasi-modular form 
of weight $6$.
\end{proof}

\end{document}